\documentclass[twoside, 11pt]{article}

\usepackage[utf8]{inputenc}
\usepackage[T1]{fontenc}
\usepackage[english]{babel}
\usepackage{amsmath, amssymb, amsthm, amstext, amsfonts, a4}
\usepackage{dsfont}
\usepackage[svgnames]{xcolor}
\usepackage{graphics,graphicx,subfigure}
\usepackage{comment}
\usepackage[affil-it]{authblk}
\usepackage{empheq}
\usepackage{lipsum}
\usepackage{capt-of}

\usepackage{hyperref}
\hypersetup{%
colorlinks=true,
linkcolor=Navy,
citecolor=Brown,
urlcolor=Navy
}

\usepackage{geometry}
\numberwithin{equation}{section}

\theoremstyle{plain}
	\newtheorem{theorem}{Theorem}[section] 
	\newtheorem{proposition}[theorem]{Proposition}       
	\newtheorem{lemma}[theorem]{Lemma}
	\newtheorem{corollary}[theorem]{Corollary}
        
\theoremstyle{definition}
	\newtheorem{definition}[theorem]{Definition}
    
	\newtheorem{remark}{Remark}[section]

\theoremstyle{remark}
    \newtheorem*{thx}{Acknowledgments}

\renewenvironment{proof}{\smallskip\noindent\emph{\textbf{Proof.}}%
  \hspace{1pt}}{\hspace{-5pt}{\nobreak\quad\nobreak\hfill\nobreak%
    $\square$\vspace{2pt}\par}\smallskip\goodbreak}
\newcommand{\limit}[2]{{\ \underset{#1 \to #2}{\longrightarrow} \ }}

\newcommand{\ds}[1]{\displaystyle{#1}}

\newcommand{\1}{\mathbf{1}} \renewcommand{\d}[1]{\mathinner{\mathrm{d}{#1}}} 
\newcommand{\p}{\partial} \newcommand{\eps}{\mathrm{\varepsilon}}

\newcommand{\N}{\mathbb{N}} \newcommand{\Z}{\mathbb{Z}} \newcommand{\R}{\mathbb{R}} 

\newcommand{\Czero}{\mathbf{C}} 
\newcommand{\Ck}[1]{\mathbf{C}^{#1}} 
\newcommand{\Cc}[1]{\mathbf{C}_\mathbf{c}^{#1}}

\newcommand{\Lip}{\mathbf{Lip}} 

\renewcommand{\L}[1]{\mathbf{L}^{#1}} 
\newcommand{\Lloc}[1]{\mathbf{L}_{\mathbf{loc}}^{#1}} 
\newcommand{\W}[2]{\mathbf{W}^{#1, #2}} 
\newcommand{\BV}{\mathbf{BV}} 
\newcommand{\BVloc}{\mathbf{BV}_{\mathbf{loc}}}
\newcommand{\TV}{\mathbf{TV}}

\newcommand{\bK}{\mathbf{K}}
\newcommand{\bF}{\mathbf{F}}
\newcommand{\bD}{\mathbf{D}}
\newcommand{\bB}{\mathbf{B}}

\newcommand{\cR}{\mathcal{R}}
\newcommand{\cP}{\mathcal{P}}

\DeclareFontFamily{U}{mathx}{\hyphenchar\font45}
\DeclareFontShape{U}{mathx}{m}{n}{
      <5> <6> <7> <8> <9> <10>
      <10.95> <12> <14.4> <17.28> <20.74> <24.88>
      mathx10
      }{}
\DeclareSymbolFont{mathx}{U}{mathx}{m}{n}
\DeclareFontSubstitution{U}{mathx}{m}{n}
\DeclareMathAccent{\widecheck}{0}{mathx}{"71}

\begin{document}

\title{\textbf{Influence of a slow moving vehicle on traffic: Well-posedness and approximation for a mildly non-local model}}

\author{Abraham Sylla$^1$}

\date{}

\maketitle

\footnotetext[1]{\texttt{Abraham.Sylla@lmpt.univ-tours.fr} \\
Institut Denis Poisson, CNRS UMR 7013, Université de Tours, Université d'Orléans \\
Parc de Grandmont, 37200 Tours cedex, France \\
ORCID number: 0000-0003-1784-4878}

\thispagestyle{empty}

\begin{abstract}
    In this paper, we propose a macroscopic model that describes the influence of a slow moving large vehicle on road traffic. The model consists of a scalar 
	conservation law with a non-local constraint on the flux. The constraint level depends on the trajectory of the slower vehicle which is given by an ODE 
	depending on the downstream traffic density. After proving well-posedness, we first build a finite volume scheme and prove its convergence, and then investigate 
	numerically this model by performing a series of tests. In particular, the link with the limit local problem of [M. L. Delle Monache and P. Goatin, 
	\emph{J. Differ. Equ.} \textbf{257} (2014), 4015--4029] is explored numerically.

\end{abstract}



\tableofcontents

\newpage

\clearpage
\pagenumbering{arabic}

\section{Introduction}

Delle Monache and Goatin developed in \cite{DMG2014} a macroscopic model aiming at describing the situation in which a slow moving large vehicle -- a bus for 
instance -- reduces the road capacity and thus generates a moving bottleneck for the surrounding traffic flow. Their model is given by a Cauchy problem for 
Lightwill-Whitham-Richards scalar conservation law in one space dimension with local point constraint. The constraint is prescribed along the slow vehicle 
trajectory $(y(t),t)$, the unknown $y$ being coupled to the unknown $\rho$ of the constrained LWR equation. Point constraints were introduced in 
\cite{CR2005, CG2007} to account for localized in space phenomena that may occur at exits and which act as obstacles. The constraint in the model of \cite{DMG2014} 
depends upon the slow vehicle speed $\dot y$, where its position $y$ verifies the following ODE
\begin{equation}
	\label{A}
	\dot y(t) = \omega \left( \rho(y(t)^{+},t) \right).
    \tag{A}
\end{equation}

Above, $\rho = \rho(x,t) \in [0,R]$ is the traffic density and $\omega: [0,R] \to \R^{+}$ is a nonincreasing Lipschitz continuous function which 
links the traffic density to the slow vehicle velocity. Delle Monache and Goatin proved an existence result for their model in \cite{DMG2014} with a 
wave-front tracking approach in the $\BV$ framework. Adjustments to the result were recently brought by Liard and Piccoli in \cite{LP2020}. Despite the step forward 
made in \cite{DMG2017}, the uniqueness issue remained open for a time. Indeed, the appearance of the trace $\rho(y(t)^{+},t)$ makes it fairly difficult to get a 
Lipschitz continuous dependency of the trajectory $y=y(t)$ from the solution $\rho=\rho(x,t)$. Nonetheless, a highly nontrivial uniqueness result was achieved by 
Liard and Piccoli in \cite{LP2018}. To describe the influence of a single vehicle on the traffic flow, the authors of \cite{LMP2011} proposed a PDE-ODE 
coupled model without constraint on the flux for which they proposed in \cite{BCLMP2018} two convergent schemes. In the present paper, we consider a modified model 
where the point constraint becomes non-local, making the velocity of the slow vehicle depend on the mean density evaluated in a small vicinity ahead the driver. 
More precisely, instead of \ref{A}, we consider the relation
\begin{equation}
	\label{B}
	\dot y(t) = \omega \left( \int_{\R} \rho(x+y(t),t) \mu(x) \d{x} \right),
    \tag{B}
\end{equation}

where $\mu \in \BV(\R, \R^{+})$ is a weight function used to average the density. From the mathematical point of view, this choice makes the study of the 
new model easier. Indeed, the authors of \cite{ADR2014, ADRR2016, ADRR2018} put forward techniques for full well-posedness analysis of similar models with non-local
point constraints. From the modeling point of view, considering \ref{B} makes sense for several reasons outlined in Section \ref{discussion}.

The paper is organized as follows. Sections \ref{Section1} and \ref{Section2} are devoted to the proof of the well-posedness of the model. In Section \ref{Section3} 
we introduce the numerical finite volume scheme and prove its convergence. An important step of the reasoning is to prove a $\BV$ regularity for the 
approximate solutions. It serves both in the existence proof, and it is central in the uniqueness argument. In that optic, the appendix is essential. Indeed, it is 
devoted to the proof of a $\BV$ regularity for entropy solutions to a large class of limited flux models. Let us stress that we highlight the 
interest of the $\BVloc$ discrete compactness technique of Towers \cite{TowersOSLC} in the context of general discontinuous-flux problems. In the numerical 
section \ref{Section4}, first we perform numerical simulations to validate our model. Then we investigate both qualitatively and quantitatively the proximity 
between our model -- in which we considered \ref{B} -- as $\delta \to \mu_{0^+}$ and the 
model of \cite{DMG2014} in which the authors considered \ref{A}.

\newpage

\section{Model, notion of solution and uniqueness}
\label{Section1}

\subsection{Model in the bus frame}

Note that we find it convenient to study the problem in the bus frame, which means setting $X = x-y(t)$ in the model of Delle Monache and Goatin in 
\cite{DMG2014}. Keeping in mind what we said above about the non-local constraint, the problem we consider takes the following form:
\begin{equation}
	\label{AS2020}
	\left\{
		\begin{array}{lcr}
			\p_{t}\rho + \p_{x} \left(F(\dot y(t),\rho)\right) = 0 & & \R \times \mathopen]0,T \mathclose[ \\[5pt]
			\rho(x,0) = \rho_o(x+y_o) & & x \in \R \\[5pt]
			\left. F(\dot y(t),\rho) \right|_{x=0} \leq Q(\dot y(t)) & & t \in \mathopen]0,T \mathclose[ \\[5pt]
			\ds{\dot y(t) = \omega \left( \int_{\R} \rho(x,t) \mu(x) \d{x} \right)} & & t \in \mathopen]0,T \mathclose[ \\[5pt]
			y(0) = y_o. & & 
		\end{array}
	\right.
\end{equation}

Above, $\rho = \rho(x,t)$ denotes the traffic density, of which maximum attainable value is $R > 0$, and 
\[
	F(\dot y(t),\rho) = f(\rho) - \dot y(t)\rho
\]

denotes the normal flux through the curve $x = y(t)$. We assume that 
$f \in \Lip([0, R], \R^+)$ is bell-shaped, which is a commonly used assumption in 
traffic dynamics:
\begin{equation}
	\label{bell_shaped}
	f(0) = f(R) = 0, \; \exists! \; \overline{\rho} \in \mathopen]0,R \mathclose[, 
    \; f'(\rho)(\overline{\rho}-\rho) > 0 \; \text{for a.e.} \; \rho \in \mathopen]0,R \mathclose[.
\end{equation}

In \cite{DMG2014}, the authors chose the function $\ds{Q(s) = \alpha \times \left(\frac{1-s}{2} \right)^2}$ to prescribe the maximal flow allowed through a 
bottleneck located at $x=0$. The parameter $\alpha \in \mathopen]0, 1\mathclose[$ was 
giving the reduction rate of the road capacity due to the presence of the slow vehicle. We use 
the $s$ variable to stress that the value of the constraint is a function of the speed of the slow vehicle. In the sequel the $s$ variable will refer to 
quantities related to the slow vehicle velocity. Regarding the function $Q$, we can allow for more general choices. Specifically,
\[
	Q : \left[0,\|\omega\|_{\L{\infty}} \right] \to \R^{+}
\]

can be any Lipschitz continuous function. It is a well known fact that in general, the total variation of an entropy solution to a constraint Cauchy problem may 
increase (see \cite[Section 2]{CG2007} for an example). However, this increase can be controlled if the constraint level does not reach the maximum level. 
A mild assumption on $Q$ (see Assumption \eqref{level_constraint_splitting} below) 
will guarantee availability of $\BV$ bounds, provided we suppose that
\[
	\rho_o \in \L{1} \cap \BV(\R, [0,R]).
\]

\subsection{Notion of solution}

Throughout the paper, we denote by
\[
	\Phi(a,b) = \text{sign}(a-b)(f(a) - f(b)) \quad \text{and} \quad 
    \Phi_{\dot y(t)}(a,b) = \Phi(a,b) - \dot y(t) |a-b|
\]

the entropy fluxes associated with the Kru{\v{z}}kov entropy $\rho \mapsto |\rho-k|$, for all $k \in [0,R]$, see \cite{Kruzhkov1970}. Following 
\cite{DMG2014, CG2007, AGS2010, CGS2013}, we give the following definition of solution for Problem \eqref{AS2020}. 

\begin{definition}
	\label{AWS}
	A couple $(\rho,y)$ with $\rho \in \L{\infty}(\R \times \mathopen]0,T \mathclose[, \R)$ and $y \in \W{1}{\infty}(\mathopen]0,T \mathclose[, \R)$ is an admissible weak solution to \eqref{AS2020} if
    
	(i) the following regularity is fulfilled:
	\begin{equation}
		\label{continuous_in_l1}
		\rho \in \Czero([0,T], \Lloc{1}(\R, \R));
	\end{equation}

	(ii) for all test functions $\varphi \in \Cc{\infty}(\R \times \R^{+}, \R^+)$ 
    and $k \in [0,R]$, the following entropy inequalities 
	are verified for all $0 \leq \tau < \tau' \leq T$:
	\begin{equation}
		\label{EI}
		\begin{aligned}
			& \int_{\tau}^{\tau'} \int_{\R} |\rho-k| \p_{t}\varphi + \Phi_{\dot y(t)}(\rho,k) \p_{x} \varphi \; \d x \d t  
			+ \int_{\R}|\rho(x,\tau)-k|\varphi(x,\tau) \d x \\
			& - \int_{\R}|\rho(x,\tau')-k|\varphi(x,\tau') \d x + 2 \int_{\tau}^{\tau'} \cR_{\dot y(t)}(k,q(t)) \varphi(0,t) \d t \geq 0,
		\end{aligned}
	\end{equation}

	where
	\[
		\cR_{\dot y(t)}(k,q(t)) :=  F(\dot y(t),k) - \min \left\{ F(\dot y(t),k),q(t) \right\} \quad \text{and} \quad q(t) := Q(\dot y(t));
	\]

	(iii) for all test functions $\psi \in \Ck{\infty}([0,T], \R^+)$ and some given 
    $\varphi \in \Cc{\infty}(\R, \R)$ which verifies $\varphi(0)=1$, the 
	following weak constraint inequalities are verified for all 
    $0 \leq \tau < \tau' \leq T$:
	\begin{equation}
		\label{CI}
		\begin{aligned}
			& - \int_{\tau}^{\tau'} \int_{\R^{+}} \rho \p_{t} (\varphi \psi) + F(\dot y(t),\rho) \p_{x}(\varphi \psi) \d x \d t 
			- \int_{\R^{+}} \rho(x,\tau) \varphi(x) \psi(\tau) \d x \\
			& + \int_{\R^{+}} \rho(x,\tau') \varphi(x) \psi(\tau') \d x \leq \int_{\tau}^{\tau'} q(t) \psi(t) \d t;
		\end{aligned}
	\end{equation}

	(iv) the following weak ODE formulation is verified for all $t \in [0,T]$:
	\begin{equation}
		\label{ODE}
		y(t) = y_o + \int_{0}^{t} \omega \left( \int_{\R} \rho(x,s) \mu(x) \d x \right) \d s.
	\end{equation}
\end{definition}

\begin{definition}
	\label{BVRS}
	We will call $\BV$-regular solution any admissible weak solution $(\rho,y)$ to the Problem \eqref{AS2020} which also verifies
	\[
		\rho \in \L{\infty}(\mathopen]0,T \mathclose[, \BV(\R, \R)).
	\]
\end{definition}

\begin{remark}
	\label{Rk1}
	It is more usual to formulate \eqref{EI} with 
    $\varphi \in \Cc{\infty}(\R \times [0, T\mathclose[)$, $\tau = 0$ and $\tau' = T$. The equivalence between the two 
	formulations is due to the regularity \eqref{continuous_in_l1}. 
\end{remark}

\begin{remark}
	\label{Rk2}
	As it happens, the time-continuity regularity \eqref{continuous_in_l1} is actually a consequence of inequalities \eqref{EI}. Indeed, we will use the result 
	\cite[Theorem 1.2]{CancesGallouet2011} which states that if $\Omega$ is an open subset of $\R$ and if for all test functions 
	$\varphi \in \Cc{\infty}(\Omega \times [0, T\mathclose[, \R^+)$ and $k \in [0,R]$, $\rho$ satisfies the following entropy inequalities:
	\[
		\int_{0}^{T} \int_{\Omega} |\rho-k| \p_{t}\varphi + \Phi_{\dot y(t)}(\rho,k) \p_{x}\varphi \; \d x \d t 
		+ \int_{\Omega}|\rho_o(x)-k|\varphi(x,0) \d x  \geq 0,
	\]

	then $\ds{\rho \in \Czero([0,T], \Lloc{1}(\Omega, \R))}$. Moreover, since $\rho$ is bounded and $\overline{\Omega} \backslash \Omega$ has a Lebesgue measure $0$, 
	$\rho \in \Czero([0,T], \Lloc{1}(\overline{\Omega}, \R))$. We will use this remark several times in the sequel of the paper, with $\Omega = \R^{*}$. 
\end{remark}

\begin{remark}
	Any admissible weak solution $(\rho,y)$ to Problem \eqref{AS2020} is also a distributional solution to the conservation law in \eqref{AS2020}. Therefore, 
	inequalities \eqref{CI} imply the following ones for all $0 \leq \tau < \tau' \leq T$:
	\[
		\begin{aligned}
			& \int_{\tau}^{\tau'} \int_{\R^{-}} \rho \p_{t}( \varphi \psi) + F(\dot y(t),\rho) \p_{x}(\varphi \psi) \d x \d t + 
			\int_{\R^{-}} \rho(x,\tau) \varphi(x) \psi(\tau) \d x \\
			& - \int_{\R^{-}} \rho(x,\tau') \varphi(x) \psi(\tau') \d x \leq \int_{\tau}^{\tau'} q(t) \psi(t) \d t,
		\end{aligned}
	\]

	where $\varphi$ and $\psi$ are such as described in Definition \ref{AWS} (iii).
\end{remark}

The interest of weak formulations \eqref{CI}-\eqref{ODE} for the flux constraint and for the ODE governing the slow vehicle lies in their stability with respect 
to $\rho$. Formulation \eqref{EI}~--~\eqref{ODE} is well suited for passage to the limit of a.e. convergent sequences of exact or approximate solutions.

\subsection{Uniqueness of the BV-regular solution}

In this section, we prove stability with respect to the initial data and uniqueness for $\BV$-regular solutions to Problem \eqref{AS2020}. We start with the

\begin{lemma}
	\label{bus_BV}
    If $(\rho,y)$ is an admissible weak solution to Problem \eqref{AS2020}, then $\dot y \in \W{1}{\infty}(\mathopen]0,T \mathclose[, \R)$. In particular, $\dot y \in \BV([0,T], \R)$.
\end{lemma}

\begin{proof}
	Denote for all $t \in [0,T]$,
	\[
		s(t) := \omega \left( \int_{\R} \rho(x,t) \mu(x) \d x \right).
	\]

	Since $\mu \in \L{1}\cap \L{\infty}(\R, \R)$ and 
    $\rho \in \Czero([0,T], \Lloc{1}(\R, \R))$, $s$ is continuous on $[0,T]$. By definition, $y$ 
	satisfies the weak ODE formulation \eqref{ODE}. Consequently, for a.e. $t \in \mathopen]0,T \mathclose[$, $\dot y(t) = s(t)$. We are going to prove that $s$ is Lipschitz continuous 
	on $[0,T]$, which will ensure that $\dot y \in \W{1}{\infty}(\mathopen]0,T \mathclose[, \R)$.
	Since $\mu \in \BV(\R, \R)$, there exists a sequence $\ds{(\mu_{n})_{n \in \N} \subset \BV \cap \Cc{\infty}(\R, \R)}$ such that:
	\[
        \| \mu_{n} - \mu \|_{\L{1}} \limit{n}{+\infty} 0 \quad \text{and} \quad \TV(\mu_{n}) \limit{n}{+\infty} \TV(\mu).
	\]

	Introduce for all $n \in \N$ and $t \in [0,T]$, the function
	\[
		\xi_{n}(t) = \int_{\R} \rho(x,t) \mu_{n}(x) \d x.
	\]

	Fix $\psi \in \Cc{\infty}(\mathopen]0,T \mathclose[, \R)$. Since $\rho$ is a distributional solution to the conservation law in \eqref{AS2020}, we have:
	\[
		\begin{aligned}
			\forall n \in \N, \quad \int_{0}^{T} \xi_{n}(t) \dot \psi(t) \d t 
			& = \int_{0}^{T} \int_\R \rho \p_{t}(\psi \mu_{n}) \d x \d t \\
			& = - \int_{0}^{T} \int_\R F(\dot y(t),\rho) \p_{x}(\psi \mu_{n}) \d x \d t \\
			& = - \int_{0}^{T} \left( \int_\R F(\dot y(t),\rho) \mu_{n}'(x) \d x \right) \psi(t) \d t,
		\end{aligned}
	\]

	which means that for all $n \in \N$, $\xi_{n}$ admits a weak derivative 
    and that for a.e. $t \in \mathopen]0,T \mathclose[$,
	\[
		\dot \xi_{n}(t) = \int_\R F(\dot y(t),\rho) \mu_{n}'(x) \d x.
	\]

	In particular, since both the sequences $(\|\mu_{n}\|_{\L{1}})_{n}$ and $(\TV(\mu_{n}))_{n}$ are bounded -- say by $C > 0$ -- we also have for 
	all $n \in \N$,
	\[
		\|\xi_{n}\|_{\L{\infty}} \leq RC \quad \text{and} \quad \|\dot \xi_{n}\|_{\L{\infty}} \leq C(\|f\|_{\L{\infty}} + \|\omega\|_{\L{\infty}}R).
	\]

    Therefore, the sequence $(\xi_{n})_{n}$ is bounded in $\W{1}{\infty}(\mathopen]0,T \mathclose[, \R)$. Now, for all $t,\tau \in [0,T]$ and $n \in \N$, the triangle 
    inequality yields: 
	\[
		\begin{aligned}
			\left| s(t) - s(\tau) \right|
            & \leq 2 \|\omega'\|_{\L{\infty}} R \|\mu_{n}-\mu\|_{\L{1}} 
            + \|\omega'\|_{\L{\infty}} \left| \int_{\R} (\rho(x,t)-\rho(x,\tau)) \mu_{n}(x) \d{x} \right| \\
			& = 2 \|\omega'\|_{\L{\infty}} R \|\mu_{n}-\mu\|_{\L{1}} + \|\omega'\|_{\L{\infty}} | \xi_{n}(t) - \xi_{n}(\tau)| \\
			& \leq 2 \|\omega'\|_{\L{\infty}} R \|\mu_{n}-\mu\|_{\L{1}} + \underbrace{C \|\omega'\|_{\L{\infty}} (\|f\|_{\L{\infty}}
			+ \|\omega\|_{\L{\infty}}R)}_{:=K} |t-\tau|.
		\end{aligned}
	\]

    Letting $n \to + \infty$, we get that for all $t,\tau \in [0,T]$, $\ds{|s(t) - s(\tau)| \leq K|t - \tau|}$, which proves that $s$ is Lipschitz continuous on 
    $[0,T]$. The proof of the statement is completed.
\end{proof}

Before stating the uniqueness result, we make the following additional assumption:
\begin{equation}
	\label{bell_shaped_F}
	\forall s \in [0,\|\omega\|_{\L{\infty}}], \ \underset{\rho \in [0,R]}{\text{argmax}} \ F(s, \rho) > 0.
\end{equation}

This ensures that for all $s \in [0,\|\omega\|_{\L{\infty}}]$, the function $F(s,\cdot)$ verifies the bell-shaped assumptions \eqref{bell_shaped_appendix}. 
For example, when considering the flux $f(\rho) = \rho(R-\rho)$, \eqref{bell_shaped_F} reduces to $\|\omega\|_{\L{\infty}} < R$, which only means that the 
maximum velocity of the slow vehicle is lesser than the maximum velocity of the cars.

\begin{theorem}
	\label{BVRS_uniqueness}
	Let $f \in \Lip([0, R], \R^)$ and assume that \eqref{bell_shaped}-\eqref{bell_shaped_F} hold. Fix $\rho_o^{1}, \rho_o^{2} \in \L{1} \cap \BV(\R, [0, R])$ 
    and $y_o^{1}, y_o^{2} \in \R$. We denote by $(\rho^{1}, y^{1})$ a $\BV$-regular 
    solution to Problem \eqref{AS2020} corresponding to initial data 
	$(\rho_o^{1}, y_o^{1})$, and by $(\rho^{2},y^{2})$ an admissible weak solution with initial data $(\rho_o^{2}, y_o^{2})$. Then there exist constants 
	$\alpha,\beta,\gamma > 0$ such that for all $t \in [0, T]$, 
	\begin{equation}
		\label{stab1}
		\|\rho^{1}(t)-\rho^{2}(t)\|_{\L{1}(\R)} 
		\leq \left( |y_o^{1}-y_o^{2}| \TV(\rho_o^{1}) + \|\rho_o^{1}-\rho_o^{2}\|_{\L{1}(\R)} \right) \exp(\alpha t)
	\end{equation}

	and
	\begin{equation}
		\label{stab2} 
		|y^1 (t) - y^2 (t)|
		\leq |y_o^1 - y_o^2 | + (\beta |y_o^1 - y_o^2| + \gamma \|\rho_o^1 - \rho_o^2 \|_{\L{1}(\R)}) (\exp(\alpha t)  - 1).
	\end{equation}

	In particular, Problem \eqref{AS2020} admits at most one $\BV$-regular solution.
\end{theorem} 

\begin{proof}
	Since $(\rho^1, y^1)$ is a $\BV$-regular solution to Problem \eqref{AS2020}, there exists $C \geq 0$ such that
	\[
		\forall t \in [0,T], \ \TV(\rho^1 (t)) \leq C.
	\]

	Lemma \ref{bus_BV} ensures that $\dot y^1, \dot y^2 \in \BV([0,T], \R^+)$. We can use \eqref{stab_appendix} to obtain that for all $t \in [0, T]$, we have 
	\begin{equation}
        \label{like_DMG}
        \|\rho^1 (t) - \rho^2 (t)\|_{\L{1}(\R)} \leq |y_o^1 - y_o^2 | \TV(\rho_o^1) + \|\rho_o^1 - \rho_o^2 \|_{\L{1}}
        + \left( 2 \|Q'\|_{\L{\infty}} + 2R + C \right) \int_{0}^{t} |\dot y^1 (s) - \dot y^2 (s)| \d s.
	\end{equation}

	Moreover, since for a.e. $t \in \mathopen]0,T \mathclose[$,
	\[
		|\dot y^1 (t) - \dot y^2 (t)| \leq \|\omega'\|_{\L{\infty}}\|\mu\|_{\L{\infty}} \|\rho^1 (t) - \rho^2 (t)\|_{\L{1}(\R)},
	\]

	Gronwall's lemma yields \eqref{stab1} with 
	$\ds{\alpha = \left( 2 \|Q'\|_{\L{\infty}} + 2R + C \right) \|\omega'\|_{\L{\infty}}\|\mu\|_{\L{\infty}}}$. Then for all $t \in [0,T]$,
	\[
		\begin{aligned}
			|y^1 (t) - y^2 (t)| 
			& \leq |y_o^1 - y_o^2| + \int_{0}^{t} |\dot y^1 (s) - \dot y^2 (s)| \d s \\
			& \leq |y_o^1 - y_o^2| + \|\omega'\|_{\L{\infty}} \|\mu\|_{\L{\infty}} \int_{0}^{t} \|\rho^1 (s) - \rho^2 (s)\|_{\L{1}(\R)} \d s \\
			& \leq |y_o^1 - y_o^2| + (\beta |y_o^1 - y_o^2| + \gamma \|\rho_o^1 - \rho_o^2 \|_{\L{1}(\R)}) (\exp(\alpha t) - 1),
		\end{aligned}
	\]

	where
	\[
        \beta = \frac{\TV(\rho_o^1)}{ 2 \|Q'\|_{\L{\infty}} + 2R + C} 
        \quad \text{and} \quad 
        \gamma = \frac{1}{2 \|Q'\|_{\L{\infty}} + 2R + C}.
	\]

	The uniqueness of a $\BV$-regular solution is then clear.
\end{proof}

\begin{remark}
	Up to inequality \eqref{like_DMG}, our proof was very much following the one of \cite[Theorem 2.1]{DMG2017}. However, the authors of \cite{DMG2017} 
	faced an issue to derive a Lipschitz stability estimate between the car densities and the slow vehicle velocities starting from
	\[
		|\omega\left(\rho^1 (0+,t)\right) - \omega\left(\rho^2 (0+,t)\right)|.
	\]

	For us, due to the non-locality of our problem, it was straightforward to obtain the bound
	\[
		\left|\omega \left(\int_{\R} \rho^1 (x,t) \mu(x) \d x \right) - \omega \left(\int_{\R} \rho^2 (x,t) \mu(x) \d x \right) \right| 
		\leq \|\omega'\|_{\L{\infty}}\|\mu\|_{\L{\infty}} \|\rho^1 (t) - \rho^2 (t)\|_{\L{1}(\R)}.
	\]
\end{remark}

\begin{remark}
	\label{weak_strong}
	A noteworthy consequence of Theorem \ref{BVRS_uniqueness} is that existence of a $\BV$-regular solution will ensure uniqueness of an admissible weak one.
\end{remark}

\section{Two existence results}
\label{Section2}

\subsection{Time-splitting technique}

In \cite{DMG2014}, to prove existence for their problem, the authors took a wave-front tracking approach. We choose here to use a time-splitting technique. 
The main advantage of this technique is that it relies on a ready-to-use theory. More precisely, at each time step, we will deal with exact solutions to 
a conservation law with a flux constraint, which have now become standard, see \cite{CG2007, AGS2010, CGS2013}.
\bigskip 

Fix $\rho_o \in \L{1}(\R, [0,R])$ and $y_o \in \R$. Let $\delta > 0 $ be a time step, $N \in \N$ such that $T \in [N \delta, (N+1)\delta \mathclose[$ and denote for 
all $n \in [\![0; N+1]\!]$, $t^{n} = n \delta$. We initialize with
\[
	\forall t \in \R, \; \rho^{0}(t) = \rho_o(\cdot +y_o) \quad \text{and} \quad \forall t \in [0,T], \ y^{0}(t) = y_o.
\]

Fix $n \in [\![1; N+1]\!]$. First, we define for all $t \in \mathopen]t^{n-1},t^{n}]$,
\[
	\sigma^{n}(t) = \omega \left( \int_{\R} \rho^{n-1}(x,t - \delta) \mu(x) \d x \right), \ s^{n} = \sigma^{n}(t^{n}) \ \text{and} \ q^{n} = Q(s^{n}).
\]

Since both $q^n$ and $\rho^{n-1}(\cdot, t^{n-1})$ are bounded, \cite[Theorem 2.11]{AGS2010} ensures the existence and uniqueness of a solution 
$\rho^n \in \L{\infty}(\R \times [t^{n-1},t^{n}], \R)$ to
\[
	\left\{
		\begin{array}{lcr}
			\p_{t}\rho + \p_{x} \left(F(s^{n},\rho) \right) = 0 & & \R \times \mathopen]t^{n-1}, t^n\mathclose[ \\[5pt]
			\rho(x,t^{n-1}) = \rho^{n-1}(x,t^{n-1}) & & x \in \R \\[5pt]
			\left. F(s^{n},\rho) \right|_{x=0} \leq q^{n} & & t \in \mathopen]t^{n-1}, t^n\mathclose[,
		\end{array}
	\right.
\]

in the sense that $\rho^n$ satisfies entropy/constraint inequalities analogous to 
\eqref{EI}-\eqref{CI} with suitable flux, constraint and initial 
datum, see Definition \ref{AWS_appendix}. Taking also into account Remark \ref{Rk2}, $\rho^{n} \in \Czero([t^{n-1},t^{n}], \Lloc{1}(\R, \R))$. We then 
define the following functions:
\[
	\begin{aligned}
		& \bullet \ \rho_{\delta}(t) = \rho^{0} \1_{\R^{-}}(t) + \sum_{n=1}^{N+1} \rho^{n}(t) \1_{\mathopen]t^{n-1},t^{n}]}(t) \\
		& \bullet \ \sigma_{\delta}(t), q_{\delta}(t), s_{\delta}(t) = \sigma^{n}(t), q^{n}, s^{n} \; \text{if} \ t \in \mathopen]t^{n-1},t^{n}] \\
		& \bullet \ y_{\delta}(t) = y_o + \int_{0}^{t} \sigma_{\delta}(u) \d u.
	\end{aligned}
\]

First, let us prove that $(\rho_{\delta},y_{\delta})$ solves an approximate version of Problem \eqref{AS2020}.

\begin{proposition}
	\label{splitting}
	The couple $(\rho_{\delta},y_{\delta})$ is an admissible weak solution to
	\begin{equation}
		\label{AS2020_splitting}
		\left\{
			\begin{array}{lcr}
				\p_{t}\rho_{\delta} + \p_{x} \left( F(s_{\delta}(t),\rho_{\delta}) \right) = 0 & & \R \times \mathopen]0,T \mathclose[ \\[5pt]
				\rho_{\delta}(x,0) = \rho_o(x+y_o) & & x \in \R \\[5pt]
				\left. F(s_{\delta}(t),\rho_{\delta}) \right|_{x=0} \leq q_{\delta}(t) & & t \in \mathopen]0,T \mathclose[ \\[5pt]
				\ds{\dot y_{\delta}(t) = \omega \left( \int_{\R} \rho_{\delta}(x,t-\delta) \mu(x) \d x \right)} & & t \in \mathopen]0,T \mathclose[ \\[5pt]
				y_{\delta}(0) = y_o, & & 
			\end{array}
		\right.
	\end{equation}

	in the sense that $\rho_\delta \in \Czero([0,T], \Lloc{1}(\R, \R))$ and satisfies entropy/constraint inequalities analogous to 
	\eqref{EI}-\eqref{CI} with flux $F(s_{\delta}(\cdot), \cdot)$, constraint $q_\delta$, and initial data $\rho_o(\cdot + y_o)$; and $y_{\delta}$ satisfies, 
	instead of \eqref{ODE}, the following weak ODE formulation:
	\[
		\forall t \in [0,T], \quad y_{\delta}(t) = y_o + \int_0^t \omega \left( \int_\R \rho_\delta(x,s-\delta) \mu(x) \d x \right) \d s.	
	\]
\end{proposition}

\begin{proof}
	By construction, for all $n \in [\![1; N+1]\!]$, 
    $\rho^{n} \in \Czero([t^{n-1},t^{n}], \Lloc{1}(\R, \R))$. Combining this with the 
    “stop-and-restart” conditions $\rho^{n}(\cdot,t^{n-1}) = \rho^{n-1}(\cdot,t^{n-1})$, one ensures that $\rho_{\delta} \in \Czero([0, T], \Lloc{1}(\R, \R))$. Let 
    $t \in [0,T]$ and $n \in [\![1; N+1]\!]$ such that 
    $t \in [t^{n-1},t^{n}\mathclose[$. Then,
	\begin{equation}
		\label{ODE_splitting}
		\begin{split}
			y_{\delta}(t) - y_o
			= & \sum_{k=1}^{n-1} \int_{t^{k-1}}^{t^{k}} \sigma^{k}(s) \d s + \int_{t^{n-1}}^{t} \sigma^{n}(s) \d s \\
			= & \sum_{k=1}^{n-1} \int_{t^{k-1}}^{t^{k}} 
            \omega \left( \int_{\R} \underbrace{\rho^{k-1}(x,s - \delta)}_{\rho_{\delta}(x,s-\delta)} \mu(x) \d x \right) \d s
             + \int_{t^{n-1}}^{t} \omega \left( \int_{\R} \underbrace{\rho^{n-1}(x,s - \delta)}_{\rho_{\delta}(x,s-\delta)} \mu(x) \d x \right) \d s \\
			= & \int_{0}^{t} \omega \left(\int_{\R} \rho_{\delta}(x,s-\delta) \mu(x) \d x \right) \d s,
		\end{split}
	\end{equation}

	which proves that $\dot y_\delta$ solves the ODE in \eqref{AS2020_splitting} in the weak sense. Fix now 
	$\varphi \in \Cc{\infty}(\R \times \R^{+}, \R^+)$ and $k \in [0,R]$. By construction of the sequence $((\rho^{k},y^{k}))_{k}$, we have for 
	all $n,m \in [\![0; N+1]\!]$,
	\[
		\begin{split}
			& \int_{t^n}^{t^m} \int_{\R} |\rho_{\delta}-k| \p_{t}\varphi + \Phi_{s_{\delta}(t)}(\rho_{\delta},k) \p_{x} \varphi \d x\d t \\
			= & \sum_{k=n+1}^{m} \int_{t^{k-1}}^{t^{k}} \int_{\R} |\rho^{k}-k| \p_{t}\varphi + \Phi_{s^{k}}(\rho^{k},k) \p_{x} \varphi \d x\d t \\
			\geq & \sum_{k=n+1}^{m} \left\{ \int_{\R}|\rho^{k}(x,t^{k})-k|\varphi(x,t^{k}) \d x 
            - \int_{\R}|\underbrace{\rho^{k}(x,t^{k-1})}_{\rho^{k-1}(x,t^{k-1})}-k|\varphi(x,t^{k-1}) \d x 
            - 2 \int_{t^{k-1}}^{t^{k}} \cR_{s^{k}}(k,q^{k}) \varphi(0,t) \d t \right\} \\
            = & \int_{\R}|\rho_{\delta}(x,t^m)-k|\varphi(x,t^m) \d x - \int_{\R}|\rho_{\delta}(x,t^n)-k|\varphi(x,t^n) \d x 
            - 2 \int_{t^{n}}^{t^{m}} \cR_{s_{\delta}(t)}(k,q_{\delta}(t)) \varphi(0,t) \d t.
		\end{split}
	\]

	It is then straightforward to prove that for all $0 \leq \tau < \tau' \leq T$,
	\begin{equation}
		\label{EI_splitting}
		\begin{aligned}
			& \int_{\tau}^{\tau'} \int_{\R} |\rho_{\delta}-k| \p_{t}\varphi + \Phi_{s_{\delta}(t)}(\rho_{\delta},k) \p_{x} \varphi \ \d x\d t 
			+ \int_{\R}|\rho_{\delta}(x,\tau)-k|\varphi(x,\tau) \d x \\
			& - \int_{\R}|\rho_{\delta}(x,\tau')-k|\varphi(x,\tau') \d x  
			+ 2 \int_{\tau}^{\tau'} \cR_{s_{\delta}(t)}(k,q_{\delta}(t)) \varphi(0,t) \d t \geq 0.
		\end{aligned}
	\end{equation}

	Proving that $\rho_\delta$ satisfies constraint inequalities is very similar so we omit the details. One has to start from
	\[
		- \int_{\tau}^{\tau'} \int_{\R^{+}} \rho_{\delta} \p_{t} (\varphi \psi) + F(s_{\delta}(t),\rho_{\delta}) \p_{x} (\varphi \psi) \ \d x\d t
	\]

	and make use once again of the construction of the sequence $((\rho^{k},y^{k}))_{k}$ to obtain
	\begin{equation}
		\label{CI_splitting}
		\begin{aligned}
			& - \int_{\tau}^{\tau'} \int_{\R^{+}} \rho_{\delta} \p_{t} (\varphi \psi) + F(s_{\delta}(t),\rho_{\delta}) \p_{x} (\varphi \psi) \ \d x\d t 
			- \int_{\R^{+}} \rho_{\delta}(x,\tau) \varphi(x) \psi(\tau) \d x \\
			& + \int_{\R^{+}} \rho_{\delta}(x,\tau') \varphi(x) \psi(\tau') \d x  \leq  \int_{\tau}^{\tau'} q_{\delta}(t) \psi(t) \d t.
		\end{aligned}
	\end{equation}

	This concludes the proof.
\end{proof}

\begin{remark}
	\label{compactness_bus_splitting}
	Remark that we have for all $\delta > 0$,
	\[
        \|\sigma_{\delta}\|_{\L{\infty}} \leq \|\omega\|_{\L{\infty}} 
        \quad \text{and} \quad 
        \|y_{\delta}\|_{\L{\infty}} \leq |y_o| + T \|\omega\|_{\L{\infty}}.
	\]

	This means that the sequence $(y_{\delta})_{\delta}$ is bounded in $\W{1}{\infty}
    (\mathopen]0,T \mathclose[, \R)$. Then the compact embedding of $\W{1}{\infty}
    (\mathopen]0,T \mathclose[, \R)$ in $\Czero([0,T], \R)$ yields a subsequence of 
    $(y_{\delta})_{\delta}$, which we do not relabel, which converges uniformly on 
    $[0,T]$ to some $y \in \Czero([0,T], \R)$.
\end{remark}

At this point, we propose two ways to obtain compactness for the sequence $(\rho_{\delta})_{\delta}$, which will lead to two existence results.

\subsection{The case of a nondegenerately nonlinear flux}\label{non_deg_flux}

\begin{theorem}
	\label{AWS_existence_splitting}
	Fix $\rho_o \in \L{1}(\R, [0;R])$ and $y_o \in \R$. Let $f \in \Lip([0, R], \R^)$ 
    and assume that \eqref{bell_shaped}-\eqref{bell_shaped_F} hold, as well as the 
    following nondegeneracy assumption
	\begin{equation}
		\label{non_deg}
		\text{for a.e.} \; s \in \mathopen]0,\|\omega\|_{\L{\infty}} \mathclose[, 
        \quad \emph{meas} \biggl(\{ \rho \in [0,R] \; : \; f'(\rho) - s = 0 \} \biggr) = 0.
	\end{equation}

	Then Problem \eqref{AS2020} admits at least one admissible weak solution.
\end{theorem}

\begin{proof}
	Condition \eqref{non_deg} combined with the obvious uniform $\L{\infty}$ bound
	\[
		\forall \delta > 0, \ \forall (x,t) \in \R \times [0,T], \ \rho_{\delta}(x,t) \in [0,R],
	\]

	and the results proved by Panov in \cite{Panov2009, Panov2010} ensure the existence of a subsequence -- which we do not relabel -- that converges in 
	$\Lloc{1}(\R^{*} \times \mathopen]0,T \mathclose[, \R)$ to some 
    $\rho \in \Lloc{1}(\R^{*} \times \mathopen]0,T \mathclose[, \R)$; and a further extraction yields the almost everywhere 
	convergence on $\R \times \mathopen]0,T \mathclose[$ and also the fact that $\rho \in \L{\infty}(\R \times \mathopen]0,T \mathclose[, [0,R])$. We now show that the couple $(\rho,y)$ constructed 
	above is an admissible weak solution to \eqref{AS2020} in the sense of Definition \ref{AWS}.
    \bigskip

	For all $\delta > 0$ and $t \in [0,T]$,
	\[
		\begin{split}
			y_{\delta}(t) - y_o
			= & \int_{0}^{t} \omega \left( \int_{\R} \rho_{\delta}(x,s - \delta) \mu(x) \d x \right) \d s \\
			= & \int_{-\delta}^{t-\delta} \omega \left( \int_{\R} \rho_{\delta}(x,s) \mu(x) \d x \right) \d s \\
			= & \int_{0}^{t} \omega \left( \int_{\R} \rho_{\delta}(x,s) \mu(x) \d x \right) \d s
			+ \left( \int_{-\delta}^{0} - \int_{t-\delta}^{t} \right) \omega \left( \int_{\R} \rho_{\delta}(x,s) \mu(x) \d x \right) \d s.
        \end{split}
	\]

	The last term vanishes as $\delta \to 0$ since $\omega$ is bounded. Then, Lebesgue 
    Theorem combined with the continuity of $\omega$ gives, 
	for all $t \in [0,T]$,
	\[
		y_{\delta}(t) \limit{\delta}{0} y_o + \int_{0}^{t} \omega \left( \int_{\R} \rho(x,s) \mu(x) \d x \right) \d s.
	\]

	This last quantity is also equal to $y(t)$ due to the uniform convergence of $(y_{\delta})_{\delta}$ to $y$. This proves that $y$ verifies \eqref{ODE}. Now, we 
	aim at passing to the limit in \eqref{EI_splitting} and \eqref{CI_splitting}. With this in mind, we prove the a.e. convergence of the sequence 
	$(\sigma_{\delta})_{\delta}$ towards $\dot y$. Since $\mu \in \BV(\R, \R)$, there exists a sequence of smooth functions 
	$(\mu_{n})_{n \in \N} \subset \BV \cap \Cc{\infty}(\R, \R)$ such that:
	\[
		\|\mu_n - \mu\|_{\L{1}} \limit{n}{+\infty} 0 \quad 
        \text{and} \quad \TV(\mu_{n}) \limit{n}{+\infty} \TV(\mu).
	\]

	Introduce for every $\delta > 0$ and $n \in \N$, the function
	\[
		\xi_{\delta}^{n}(t) := \int_{\R} \rho_{\delta}(x,t) \mu_{n}(x) \d x.
	\]

	Since for all $\delta > 0$, $\rho_{\delta}$ is a distributional solution to the conservation law in \eqref{AS2020_splitting}, one can show -- following the 
	proof of Lemma \ref{bus_BV} for instance -- that for every $n \in \N$, $\xi_{\delta}^{n} \in \W{1}{\infty}(\mathopen]0,T \mathclose[, \R)$, and that for a.e. $t \in \mathopen]0,T \mathclose[$,
	\[
		\dot \xi_{\delta}^{n}(t) 
        = \int_{\R} F(s_{\delta}(t),\rho_{\delta}) \mu_{n}'(x) \d x.
	\]

	Moreover, since both the sequences $\left( \|\mu_{n}\|_{\L{1}} \right)_{n}$ and $(\TV(\mu_{n}))_{n}$ are bounded, it is clear that 
	$\left(\xi_{\delta}^{n}\right)_{\delta,n}$ is uniformly bounded in $\W{1}{\infty}(\mathopen]0,T \mathclose[, \R)$, therefore so is 
    $\left(\omega (\xi_{\delta}^{n})\right)_{\delta,n}$. Consequently, for all $n \in \N$, $\delta > 0$ and almost every $t \in \mathopen]0,T \mathclose[$, the 
    triangle inequality yields:
	\[
		\begin{aligned}
			\left| \sigma_{\delta}(t) - \omega\left( \int_{\R} \rho(x,t) \mu(x) \d x \right) \right|
			& \leq 2 \|\omega'\|_{\L{\infty}} R \|\mu_{n}-\mu\|_{\L{1}} + \delta \ \sup_{n \in \N} \ \| \omega (\xi_{\delta}^{n})\|_{\W{1}{\infty}(\mathopen]0,T \mathclose[)} \\
            & + \|\omega'\|_{\L{\infty}} \left| \int_{\R} (\rho_{\delta}(x,t) - \rho(x,t)) \mu(x) \d x \right| 
            \underset{\underset{n \to +\infty}{\delta \to 0}}{\longrightarrow} 0,
		\end{aligned}
	\]

    which proves that $(\sigma_{\delta})_{\delta}$ converges a.e. on $\mathopen]0,T \mathclose[$ to $\dot y$. 
    
    To prove the time-continuity regularity, we first apply inequality 
    \eqref{EI_splitting} with $\tau = 0$, $\tau' = T$ (which is licit since $\rho_\delta$ is continuous in time), 
    $\varphi \in \Cc{\infty}(\R^{*} \times [0, T\mathclose[, \R^+)$ and $k \in [0,R]$:
	\[
		\int_{0}^{T} \int_{\R} |\rho_{\delta}-k| \p_{t}\varphi + \Phi_{\sigma_{\delta}(t)}(\rho_{\delta},k) \p_{x} \varphi \; \d x\d t
		+ \int_{\R}|\rho_o(x+y_o)-k|\varphi(x,0) \d x \geq 0.
	\]

	Then, we let $\delta \to 0$ to get
	\[
		\int_{0}^{T} \int_{\R} |\rho-k| \p_{t}\varphi + \Phi_{\dot y(t)}(\rho,k) \p_{x} \varphi \; \d x\d t 
		+ \int_{\R}|\rho_o(x+y_o)-k|\varphi(x,0) \d x \geq 0.
	\]

	Consequently, $\rho \in \Czero([0,T], \Lloc{1}(\R, \R))$, see Remark \ref{Rk2}.

    Finally, the a.e. convergences of $(\sigma_{\delta})_{\delta}$ and 
	$(\rho_{\delta})_\delta$ to $\dot y$ and $\rho$, respectively, are enough to pass 
    to the limit in \eqref{EI_splitting}. This ensures that for all test 
	functions $\varphi \in \Cc{\infty}(\R \times \R^{+}, \R^+)$ and $k \in [0,R]$, 
    the following inequalities hold for a.e. $0 \leq \tau < \tau' \leq T$:
	\[
		\begin{aligned}
			& \int_{\tau}^{\tau'} \int_{\R} |\rho-k| \p_{t}\varphi + \Phi_{\dot y(t)}(\rho,k) \p_{x} \varphi \ \d x\d t 
			+ \int_{\R}|\rho(x,\tau)-k|\varphi(x,\tau) \d x \\
			& - \int_{\R}|\rho(x,\tau')-k|\varphi(x,\tau') \d x  + 2 \int_{\tau}^{\tau'} \cR_{\dot y(t)}(k,q(t)) \varphi(0,t) \d t \geq 0.
		\end{aligned}
	\]

	Observe that the expression in the left-hand side of the previous inequality is a continuous function of $(\tau,\tau')$ which is almost everywhere greater than 
	the continuous function $0$. By continuity, this expression is everywhere greater than $0$, which proves that $\rho$ satisfies the entropy inequalities 
	\eqref{EI}. Using similar arguments, we show that $\rho$ satisfies the constraint inequalities \eqref{CI}. This proves the couple $(\rho,y)$ is an 
	admissible weak solution to Problem \eqref{AS2020}, and this concludes the proof.
\end{proof}

In this section, we proved an existence result for $\L{\infty}$ initial data, but we have no guarantee of uniqueness since \textit{a priori} we have no 
information regarding the $\BV$ regularity of such solutions. \\
Assumption \eqref{non_deg} ensures the compactness for sequences of entropy solutions to conservation laws with flux function $F$. However, it prevents us from
using flux functions with linear parts -- which corresponds to constant traffic velocity for small densities -- whereas such fundamental diagrams are often
used in traffic modeling. The results of the next section will extend to this interesting case, under the extra $\BV$ assumption on the data.

\subsection{Well-posedness for BV data}

To obtain compactness for $(\rho_\delta)_\delta$, an alternative to the setting of Section \ref{non_deg_flux} is to derive uniform $\BV$ bounds.

\begin{theorem}
	\label{BVRS_WP_splitting}
	Fix $\rho_o \in \L{1} \cap \BV(\R, [0,R])$ and $y_o \in \R$. Let 
    $f \in \Lip([0, R], \R^)$ and assume that \eqref{bell_shaped}-\eqref{bell_shaped_F} 
    hold. Suppose also that 
	\begin{equation}
		\label{reg_flux}
		\forall s \in \left[0, \|\omega\|_{\L{\infty}} \right], \quad 
        \rho \mapsto F(s, \rho) \in \Ck{1}([0,R] \backslash \{ \overline{\rho}_s\}, \R),
	\end{equation}

	where $\overline{\rho}_s = \underset{\rho \in [0,R]}{\emph{argmax}} \; F(s,\rho)$. 
    Finally, assume that $Q$ satisfies the condition
	\begin{equation}
		\label{level_constraint_splitting}
		\exists \eps > 0, \; \forall s \in [0,\|\omega\|_{\L{\infty}}], \quad 
        Q(s) \leq \underset{\rho \in [0,R]}{\emph{max}} \ F(s,\rho) - \eps.
	\end{equation}

    Then Problem \eqref{AS2020} admits a unique admissible weak solution, which is also 
    $\BV$-regular.
\end{theorem}

\begin{proof}
	Fix $\delta > 0$. Recall that $(\rho_{\delta},y_{\delta})$ is an admissible weak solution to \eqref{AS2020_splitting}. In particular, $\rho_\delta$ is an 
	admissible weak solution to the constrained conservation law in \eqref{AS2020_splitting}, in the sense of Definition \ref{AWS_appendix}. It is clear from the 
	splitting construction that for a.e. $t \in \mathopen]0,T \mathclose[$,
	\[
		\sigma_{\delta}(t) = \omega \left( \int_\R \rho_\delta (x,t-\delta) \mu(x) \d x \right).
	\]

	Following the steps of the proof of Lemma \ref{bus_BV}, we can show that for all $\delta > 0$, $\sigma_{\delta} \in \BV([0,T], \R^+)$. Even more than that, by 
	doing so we show that the sequence $\left( \TV(\sigma_{\delta}) \right)_{\delta}$ is bounded. Therefore, the sequence 
	$\left( \TV(s_{\delta}) \right)_{\delta}$ is bounded as well. Moreover, since $Q$ verifies \eqref{level_constraint_splitting}, all the hypotheses of 
	Corollary \ref{BVRS_WP_appendix} are fulfilled. Combining this with Remark \ref{weak_strong_appendix}, we get the existence of a constant 
	$C_\eps$ depending on $\| \p_s F \|_{\L{\infty}}$ such that for all $t \in [0,T]$,
	\begin{equation}
        \label{BV_bound_spliting}
        \begin{aligned}
			\TV(\rho_{\delta}(t)) 
			& \leq \TV(\rho_o) + 4R + C_\eps \left( \TV(q_\delta) + \TV(s_{\delta}) \right) \\
            & \leq \TV(\rho_o) + 4R + C_\eps (1 + \|Q'\|_{\L{\infty}}) \TV(s_{\delta}).
        \end{aligned}
	\end{equation}

	Consequently, for all $t \in [0,T]$, the sequence $(\rho_{\delta}(t))_{\delta}$ is bounded in $\BV(\R)$. A classical analysis argument -- see 
	\cite[Theorem A.8]{HRBook} -- ensures the existence of 
    $\rho \in \Czero([0,T], \Lloc{1}(\R, \R))$ such that
	\[
		\forall t \in [0,T], \; \rho_{\delta}(t) \limit{\delta}{0} \rho(t) \; \text{in} \; \Lloc{1}(\R, \R).
	\]

	With this convergence, we can follow the proof of Theorem \ref{AWS_existence_splitting} to show that $(\rho,y)$ is an admissible weak solution to 
	\eqref{AS2020}. Then, when passing to the limit in \eqref{BV_bound_spliting}, the lower semi-continuity of the $\BV$ semi-norm ensures that $(\rho,y)$ is also 
	$\BV$-regular. By Remark \ref{weak_strong}, it ensures uniqueness and concludes the proof.
\end{proof}

\subsection{Stability with respect to the weight function}

To end this section, we now study the stability of Problem \eqref{AS2020} with respect to the weight function $\mu$. More precisely, let 
$\left(\mu^{\ell}\right)_{\ell} \subset \BV(\R, \R^{+})$ be a sequence of weight functions that converges to $\mu$ in the weak $\L{1}$ sense:
\begin{equation}
	\label{weak_L1_convergence}
	\forall g \in \L{\infty}(\R, \R), \quad 
    \int_{\R} g(x) \mu^{\ell}(x) \d x \limit{\ell}{+\infty} \int_{\R} g(x) \mu(x) \d x.
\end{equation}

Let $(y_o^{\ell})_{\ell} \subset \R$ be a sequence of real numbers that converges to some $y_o$ and let $(\rho_o^{\ell})_{\ell} \subset \L{1}(\R, [0,R])$ 
be a sequence of initial data that converges to $\rho_o$ in the strong $\L{1}$ sense. We suppose that the flux function $f$ satisfies Assumptions 
\eqref{bell_shaped}-\eqref{bell_shaped_F}-\eqref{non_deg}. Theorem \ref{AWS_existence_splitting} allows us to define or all $\ell \in \N$, the couple 
$(\rho^{\ell},y^{\ell})$ as an admissible weak solution to the problem
\[
    \left\{
        \begin{array}{lcr}
            \p_{t}\rho^{\ell} + \p_{x} \left(F(\dot y^{\ell}(t),\rho^{\ell})\right) = 0 & & \R \times \mathopen]0,T \mathclose[ \\[5pt]
            \rho^{\ell}(x,0) = \rho_o^{\ell}(x+y_o^{\ell}) & & x \in \R \\[5pt]
            \left. F(\dot y^{\ell}(t),\rho^{\ell}) \right|_{x=0} \leq Q(\dot y^{\ell}(t)) & & t \in \mathopen]0,T \mathclose[ \\[5pt]
            \ds{\dot y^{\ell}(t) = \omega \left( \int_{\R} \rho^{\ell}(x,t) \mu^{\ell}(x) \d x \right)} & & t \in \mathopen]0,T \mathclose[ \\[5pt]
            y^{\ell}(0) = y_o^{\ell}. & & 
        \end{array}
    \right.
\]

\begin{remark}
	Using the same arguments as in Remark \ref{compactness_bus_splitting} and as in the proof of Theorem \ref{AWS_existence_splitting}, we get that up to the 
	extraction of a subsequence, $(y^{\ell})_{\ell}$ converges uniformly on $[0,T]$ to some $y \in \Czero([0,T], \R)$ and $(\rho^{\ell})_{\ell}$ converges a.e. on 
	$\R \times \mathopen]0,T \mathclose[$ to some $\rho \in \L{\infty}(\R \times \mathopen]0,T \mathclose[, \R)$. 
\end{remark}

\begin{theorem}
	\label{AWS_stability}
	The couple $(\rho,y)$ constructed above is an admissible weak solution to Problem \eqref{AS2020}.
\end{theorem}

\begin{proof} 
	The sequence $(\mu^{\ell})_{\ell}$ converges in the weak $\L{1}$ sense and is bounded in $\L{1}(\R)$; by the Dunford-Pettis Theorem, this sequence is 
	equi-integrable:
	\begin{equation}
		\label{equi1}
		\forall \eps > 0, \ \exists \alpha > 0, \ \forall A \in \mathcal{B}(\R), \ \text{mes}(A) < \alpha \implies \ \forall \ell \in \N, \ 
		\int_{A} \mu^{\ell}(x) \d x \leq \eps
	\end{equation}

	and
	\begin{equation}
		\label{equi2}
		\forall \eps > 0, \ \exists X > 0, \ \forall \ell \in \N, \ \int_{|x| \geq X} \mu^{\ell}(x) \d x \leq \eps.
	\end{equation}

	Fix $t \in \mathopen]0,T \mathclose[$ and $\eps > 0$. Fix $\alpha,X > 0$ given by \eqref{equi1} and \eqref{equi2}. Egoroff Theorem yields the existence of a 
    measurable subset $E_{t} \subset [-X,X]$ such that  
	\[
		\text{meas}([-X,X] \backslash E_{t}) < \alpha \; \text{and} \; \rho^{\ell}(t) \longrightarrow \rho(t) \; \text{uniformly on} \; E_{t}.
	\]

	For a sufficiently large $\ell \in \N$,
	\[ 
		\begin{split}
			& \left| \int_{\R}\rho^{\ell}(x,t) \mu^{\ell}(x) \d x - \int_{\R}\rho(x,t) \mu(x) \d x \right| \\
			\leq & \int_{|x| \geq X} |\rho^{\ell}-\rho| \mu^{\ell} \d x + \left| \int_{E_{t}} (\rho^{\ell}-\rho) \mu^{\ell} \d x \right| 
			+ \left| \int_{[-X,X] \backslash E_{t}}(\rho^{\ell}-\rho) \mu^{\ell} \d x \right| \\
			& + \left| \int_{\R}\rho \mu^{\ell} \d x - \int_{\R}\rho \mu \d x \right| \\
			\leq & R \eps + \|\rho^{\ell}-\rho\|_{\L{\infty}(E_{t})} \int_{E_{t}} \mu^{\ell}(x)\d x 
            + R \int_{[-X,X] \backslash E_{t}} \mu^{\ell}(x)\d x + \eps \\
            \leq & 2(R+1)\eps, 
		\end{split} 
	\]

	which proves that for a.e. $t \in \mathopen]0,T \mathclose[$,
	\begin{equation}
		\label{AWS_stability_convergence}
		\int_{\R} \rho^{\ell}(x,t) \mu^{\ell}(x) \d x \limit{\ell}{+\infty} \int_{\R} \rho(x,t) \mu(x) \d x.
	\end{equation}

	We get that $y$ verifies the weak ODE formulation \eqref{ODE} by passing to the limit in
	\[
		y^{\ell}(t) = y_o^{\ell} + \int_{0}^{t} \omega \left( \int_{\R} \rho^{\ell}(x,s) \mu^{\ell}(x) \d x \right) \d s.
	\]

	By definition, for all $\ell \in \N$, the couple $(\rho^{\ell},y^{\ell})$ satisfies the analogue of entropy/constraint inequalities \eqref{EI}-\eqref{CI} with 
	suitable flux/constraint functions. Applying these inequalities with $\tau = 0$, $\tau' = T$, $\varphi \in \Cc{\infty}(\R^{*} \times [0, T\mathclose[, \R^+)$ 
    and $k \in [0,R]$, we get
	\[
		\int_{0}^{T} \int_{\R} |\rho^{\ell}-k| \p_{t}\varphi + \Phi_{\dot y^{\ell}(t)}(\rho^{\ell},k) \p_{x}\varphi \d x \d t
		+ \int_{\R}|\rho_o^{\ell}(x+y_o^{\ell})-k|\varphi(x,0) \d x \geq 0.
	\]

	The continuity of $\omega$ and the convergence \eqref{AWS_stability_convergence} ensure that $(\dot y^{\ell})_{\ell}$ converges a.e. to $\dot y$. This combined 
	with the a.e. convergence of $(\rho^{\ell})_{\ell}$ to $\rho$ and Riesz-Frechet-Kolmogorov Theorem -- $\left( \rho_o^{\ell} \right)_{\ell}$ being strongly 
	compact in $\L{1}(\R, \R)$ -- is enough to show that when letting $\ell \to +\infty$ in the inequality above, we get, up to the extraction of a subsequence, that
	\[
		\int_{0}^{T} \int_{\R} |\rho-k| \p_{t}\varphi + \Phi_{\dot y(t)}(\rho,k) \p_{x} \varphi \d x \d t 
		+ \int_{\R}|\rho_o(x+y_o)-k|\varphi(x,0) \d x \geq 0.
	\]

	Consequently, $\rho \in \Czero([0,T], \Lloc{1}(\R, \R))$, see Remark \ref{Rk2}. Finally, the combined a.e. convergences of $(\dot y^{\ell})_{\ell}$ and 
    $(\rho^{\ell})_{\ell}$ to $\dot y$ and $\rho$, respectively, guarantee that $(\rho,y)$ verifies inequalities \eqref{EI}-\eqref{CI} for almost every 
    $0 \leq \tau < \tau' \leq T$. The same continuity argument we used in the proof Theorem \ref{AWS_existence_splitting} holds here to ensure that $(\rho,y)$ 
    actually satisfies the inequalities for all $0 \leq \tau < \tau' \leq T$. This concludes the proof of our stability claim.
\end{proof}

\subsection{Discussion}\label{discussion}

The last section concludes the theoretical analysis of Problem \eqref{AS2020}. The non-locality in space of the constraint delivers an easy proof of stability with 
respect to the initial data in the $\BV$ framework. Although a proof of existence using the Fixed Point Theorem was possible (\textit{cf.} \cite{ADRR2018}), we chose 
to propose a proof based on a time-splitting technique. The stability with respect to $\mu$ is a noteworthy feature, which shows a certain sturdiness of the model. 
However, the case we had in mind -- namely $\mu \to \delta_{0^{+}}$ -- is not reachable with the assumptions we used to prove the stability, especially 
\eqref{weak_L1_convergence}. We will explore this singular limit numerically, after having built a robust convergent numerical scheme for Problem 
\eqref{AS2020}. Let us also underline that unlike in \cite{LP2018, LP2020} where the authors required a particular form for the function $\omega$ to prove 
well-posedness for their model, our result holds as long as $\omega$ is Lipschitz continuous.
\bigskip 

As evoked earlier, the non-locality in space of the constraint makes the mathematical study of the model easier. But in the modeling point of view, this choice also
makes sense for several reasons. First, one can think that the velocity $\dot y$ of the slow moving vehicle -- unlike its acceleration -- is a rather 
continuous value. Even if the driver of the slow vehicle suddenly applies the brakes, the vehicle will not decelerate instantaneously. Note that the LWR model 
allows for discontinuous averaged velocity of the agents, however while modeling the slow vehicle we are concerned with an individual agent and can model its 
behavior more precisely. Moreover, considering the mean value of the traffic density in a vicinity ahead of the driver could be seen at taking into account both 
the driver anticipation and a psychological effect. For example, if the driver sees -- several dozens of meters ahead of him/her -- a speed reduction on traffic, 
he/she will start to slow down. This observation can be related to the fact that, compared to the fluid mechanics models where the typical number of agents is 
governed by the Avogadro constant, in traffic models the number of agents is at least $10^{20}$ times less. Therefore, a mild non-locality (evaluation of the 
downstream traffic flow via averaging over a handful of preceding cars) is a reasonable assumption in the macroscopic traffic models inspired by fluid mechanics.
This point of view is exploited in the model of \cite{CFGGK2019}. Note that it is feasible to substitute the basic LWR equation on $\rho$ by the non-local LWR
introduced in \cite{CFGGK2019} in our non-local model for the slow vehicle. Such mildly non-local model remains close to the basic local model of \cite{DMG2014}. It 
can be studied combining the techniques of \cite{CFGGK2019} and the ones we developed 
in this section.

\section{Numerical approximation of the model}
\label{Section3}

In this section, we aim at constructing a finite volume scheme and at proving its convergence toward the $\BV$-regular solution to \eqref{AS2020}. We will use 
the notations:
\[
	a \vee b := \max\{a,b\} \quad \text{and} \quad a \wedge b := \min\{a,b\}.
\]

Fix $\rho_o \in \L{1}(\R, [0,R])$ and $y_o \in \R$.

\subsection{Finite volume scheme in the bus frame}

For a fixed spatial mesh size $\Delta x$ and time mesh size $\Delta t$, let $x_{j} = j\Delta x $, $t^{n} = n \Delta t$. We define the grid cells 
$\bK_{j+1/2} := \mathopen]x_{j},x_{j+1}\mathclose[$. Let $N \in \N$ such that $T \in [t^N, t^{N+1}\mathclose[$. We write
\[
    \R \times [0,T] \subset \bigcup_{n=0}^{N} \bigcup_{j \in \Z} \cP_{j+1/2}^{n}, 
    \quad \cP_{j+1/2}^{n} := \bK_{j+1/2} \times [t^{n},t^{n+1}\mathclose[.
\]

We choose to discretize the initial data $\rho_o(\cdot + y_o)$ and the weight function $\mu$ with $\left(\rho_{j+1/2}^{0}\right)_{j \in \Z}$ and 
$\left(\mu_{j+1/2}\right)_{j \in \Z}$ where for all $j \in \Z$, $\rho_{j+1/2}^{0}$ and $\mu_{j+1/2}$ are their mean values on the cell $\bK_{j+1/2}$. 

\begin{remark}
	Others choice could be made, for instance in the case 
    $\rho_o \in \Czero(\R, \R)$ such that $\ds{\lim_{|x| \to +\infty} \rho_o(x)}$ exists (in which case, the 
	limit is zero due to the integrability assumption), the values $\rho_{j+1/2}^{0} = \rho_o\left(x_{j+1/2} + y_o\right)$ can be used. The only 
	requirements are 
		\[
            \forall j \in \Z, \ \rho_{j+1/2}^{0} \in [0,R] 
            \quad \text{and} \quad 
			\rho_{\Delta}^0 = \sum_{j \in \Z} \rho_{j+1/2}^{0} \1_{\bK_{j+1/2}} 
			\limit{\Delta x}{0} \rho_o(\cdot + y_o) \ \text{in} \ \Lloc{1}(\R, \R).
		\]
\end{remark}

Fix $n \in [\![0; N-1]\!]$. At each time step we first define an approximate velocity of the slow vehicle $s^{n+1}$ and a constraint level $q^{n+1}$:
\begin{equation}
	\label{speed_update}
	s^{n+1} = \omega \left( \sum_{j \in \Z} \rho_{j+1/2}^{n} \mu_{j+1/2} \Delta x \right), \quad q^{n+1} = Q\left(s^{n+1}\right).
\end{equation}

With these values, we update the approximate traffic density with the marching formula for all $j \in \Z$:
\begin{equation}
	\label{MF}
	\rho_{j+ 1/2}^{n+1} = \rho_{j+1/2}^{n} - \frac{\Delta t}{\Delta x} 
	\left( \bF_{j+1}^{n+1}(\rho_{j+1/2}^{n},\rho_{j+3/2}^{n}) - \bF_{j}^{n+1}(\rho_{j-1/2}^{n},\rho_{j+1/2}^{n}) \right),
\end{equation}

where, following the recipe of \cite{AGS2010, CGS2013},
\begin{equation}
	\label{num_flux}
	\bF_{j}^{n+1}(a,b) = 
	\left\{ 
		\begin{array}{cl}
			\bF^{n+1}(a,b) & \text{if} \ j \neq 0 \\ 
			\min\left\{ \bF^{n+1}(a,b),q^{n+1} \right\} & \text{if} \ j = 0, 
		\end{array}
	\right.
\end{equation}

$\bF^{n+1}$ being a monotone consistent and Lipschitz numerical flux associated to 
$\rho \mapsto F(s^{n+1}, \rho)$. We will also use the notation
\begin{equation}
	\label{scheme}
	\rho_{j+1/2}^{n+1} = H_{j}^{n+1}(\rho_{j-1/2}^{n},\rho_{j+1/2}^{n},\rho_{j+3/2}^{n}),
\end{equation}

where $H_{j}^{n+1}$ is given by the expression in the right-hand side of \eqref{MF}. We then define the functions
\[
    \begin{aligned}
        & \bullet \ \rho_{\Delta}(x,t) = \sum_{n=0}^{N} \sum_{j \in \Z}  \rho_{j+1/2}^{n} \1_{\cP_{j+1/2}^{n}}(x,t) \\
        & \bullet \ s_{\Delta}(t), q_{\Delta}(t) = s^{n+1}, q^{n+1} \; \text{if} \ t \in [t^{n},t^{n+1}\mathclose[ \\
        & \bullet \ y_{\Delta}(t) = y_o + \int_{0}^{t} s_{\Delta} (u) \d u.
    \end{aligned}
\]

Let $\Delta = (\Delta x, \Delta t)$. For our convergence analysis, we will assume that $\Delta \to 0$, with $\lambda = \Delta t / \Delta x$ verifying the 
CFL condition
\begin{equation}
	\label{CFL}
	\lambda \underbrace{\sup_{t \in \R^+} \  
	\left( \left\| \p_1 \bF^t \right\|_{\L{\infty}} + \left\| \p_2 \bF^t \right\|_{\L{\infty}} \right)}_{:=L} \leq 1,
\end{equation}

where $\bF^t = \bF^t(a, b)$ is the numerical flux, associated to $\rho \mapsto F(s(t),\rho)$, we use in \eqref{MF}.

\begin{remark}
	When considering the Rusanov flux or the Godunov one, \eqref{CFL} is guaranteed when
	\[
		2 \lambda ( \|f'\|_{\L{\infty}} + \|\omega\|_{\L{\infty}} ) \leq 1.
	\]
\end{remark}

\subsection{Stability and discrete entropy inequalities}

\begin{proposition}[$\L{\infty}$ stability]
	The scheme \eqref{scheme} is
    
	(i) monotone: for all $n \in [\![0; N-1]\!]$ and $j \in \Z$, $H_{j}^{n+1}$ is nondecreasing with respect to its three arguments;

	(ii) stable: 
	\begin{equation}
		\label{stab_scheme}
		\forall n \in [\![0; N]\!], \ \forall j \in \Z, \quad \rho_{j+1/2}^{n} \in [0,R].
	\end{equation}
\end{proposition}

\begin{proof}
	\textit{(i)} In the classical case, that is when $j \notin \{-1,0\}$, we simply 
	differentiate the Lipschitz function $H^{n+1}_{j}$ and make use of both the CFL 
	condition \eqref{CFL} and the monotonicity of $\bF^{n+1}$. For $j \in \{-1,0\}$, 
	note that the authors of \cite{AGS2010} pointed out (in Proposition 4.2) that 
	the modification done in the numerical flux \eqref{num_flux} does not change the 
	monotonicity of the scheme.

	\textit{(ii)} The $\L{\infty}$ stability is a consequence of the monotonicity and 
	also of the fact that $0$ and $R$ are stationary solutions of the 
	scheme. Indeed, as in \cite[Proposition 4.2]{AGS2010} for all $n \in [\![0; N]\!]$ 
	and $j \in \Z$,
	\[
		H_{j}^{n+1}(0,0,0) = 0, \quad  H_{j}^{n+1}(R,R,R) = R.
	\]
\end{proof}

In order to show that the limit of $(\rho_{\Delta})_{\Delta}$ -- under the a.e. convergence up to a subsequence -- is a solution of the conservation law in 
\eqref{AS2020}, we derive discrete entropy inequalities. These inequalities also contain terms that will help to pass to the limit in the constrained 
formulation of the conservation law, as soon as the sequence $(q_{\Delta})_\Delta$ of constraints is proved convergent as well.

\begin{proposition}[Discrete entropy inequalities]
	The numerical scheme \eqref{scheme} fulfills the following inequalities for all $n \in [\![0; N-1]\!]$, $j \in \Z$ and $k \in [0,R]$:
	\begin{equation}
		\label{DEI}
		\begin{aligned}
			& \left( |\rho_{j+1/2}^{n+1}-k| -|\rho_{j+1/2}^{n}-k| \right) \Delta x 
			+ \left( \Phi_{j+1}^{n} - \Phi_{j}^{n} \right) \Delta t \\
			& \leq \cR_{s^{n+1}}(k,q^{n+1}) \Delta t \ \delta_{j \in \{-1,0\}} 
			+ \left( \Phi_{0}^{n} - \Phi_{int}^{n} \right) \Delta t \left( \delta_{j = -1} - \delta_{j = 0} \right),
		\end{aligned}
	\end{equation}

	where $\Phi_{j}^{n}$ and $\Phi_{int}^{n}$ denote the numerical fluxes:
	\[
		\begin{aligned}
			& \Phi_{j}^{n} := 
			\bF^{n+1}(\rho_{j-1/2}^{n} \vee k,\rho_{j+1/2}^{n} \vee k) 
			- \bF^{n+1}(\rho_{j-1/2}^{n} \wedge k,\rho_{j+1/2}^{n} \wedge k), \\
			& \Phi_{int}^{n} := \min\{\bF^{n+1}(\rho_{-1/2}^{n} \vee k,\rho_{1/2}^{n} \vee k),q^{n+1}\}
			- \min\{\bF^{n+1}(\rho_{-1/2}^{n} \wedge k,\rho_{1/2}^{n} \wedge k),q^{n+1}\}
		\end{aligned}
	\]

	and
	\[
		\cR_{s^{n+1}}(k,q^{n+1}) := F(s^{n+1},k) - \min\{ F(s^{n+1},k),q^{n+1}\}.
	\]
\end{proposition}

\begin{proof}
	This result is a direct consequence of the scheme monotonicity. When the constraint does not enter the calculations \textit{i.e.} $j \notin \{-1,0\}$, the proof 
	follows \cite[Lemma 5.4]{EGHBook}. The key point is not only the monotonicity, but also the fact that in the classical case, all the constants 
	$k \in [0,R]$ are stationary solutions of the scheme. This observation does not hold when the constraint enters the calculations. For example if $j = -1$,
	\[
		H_{-1}^{n+1}(k,k,k) = k + \lambda \cR_{s^{n+1}}(k,q^{n+1}).
	\]

	Consequently, we have both
	\[
		\rho_{-1/2}^{n+1} \vee k  
		\leq H_{-1}^{n+1}(\rho_{-3/2}^{n} \vee k,\rho_{-1/2}^{n} \vee k,\rho_{1/2}^{n} \vee k)
	\]

	and
	\[
		\rho_{-1/2}^{n+1} \wedge k 
		\geq H_{-1}^{n+1}(\rho_{-3/2}^{n} \wedge k,\rho_{-1/2}^{n} \wedge k,\rho_{1/2}^{n} \wedge k)
		- \lambda \cR_{s^{n+1}}(k,q^{n+1}).
	\]

	By substracting these last two inequalities, we get
	\[
		\begin{split}
            |\rho_{-1/2}^{n+1} - k|
            & = \rho_{-1/2}^{n+1} \vee k - \rho_{-1/2}^{n+1} \wedge k \\
			& \leq H_{-1}^{n+1}(\rho_{-3/2}^{n} \vee k,\rho_{-1/2}^{n} \vee k,\rho_{1/2}^{n} \vee k)
            - H_{-1}^{n+1}(\rho_{-3/2}^{n} \wedge k,\rho_{-1/2}^{n} \wedge k,\rho_{1/2}^{n} \wedge k) \\
			& + \lambda \cR_{s^{n+1}}(k,q^{n+1}) \\
			& = |\rho_{-1/2}^{n} - k| 
			- \lambda \left( \min\{\bF^{n+1}(\rho_{-1/2}^{n} \vee k,\rho_{1/2}^{n} \vee k),q^{n+1}\} 
			- \bF^{n+1}(\rho_{-1/2}^{n} \vee k,\rho_{1/2}^{n} \vee k) \right) \\
			& + \lambda \left( \min\{\bF^{n+1}(\rho_{-1/2}^{n} \wedge k,\rho_{1/2}^{n} \wedge k),q^{n+1}\} 
            - \bF^{n+1}(\rho_{-1/2}^{n} \wedge k,\rho_{1/2}^{n} \wedge k) \right) + \lambda \cR_{s^{n+1}}(k,q^{n+1}) \\
			& = |\rho_{-1/2}^{n} - k| 
			- \lambda \left(\Phi_{0}^{n} - \Phi_{-1}^{n}\right) + \lambda \left(\Phi_{0}^{n} - \Phi_{int}^n \right) 
			+ \lambda \cR_{s^{n+1}}(k,q^{n+1}),
		\end{split}
	\]

	which is exactly \eqref{DEI} in the case $j=-1$. The case $j=0$ is similar, so we omit the details of the proof for this case.
\end{proof}

Starting from \eqref{MF} and \eqref{DEI}, we can obtain approximate versions of \eqref{EI} and \eqref{CI}. Let us introduce the functions:
\[
	\Phi_{\Delta}\left( \rho_{\Delta},k \right) := \sum_{n=0}^{N} \sum_{j \in \Z} \Phi_{j}^{n} \1_{\cP_{j+1/2}^{n}}, \quad
	\bF_{\Delta}(s_{\Delta},\rho_{\Delta}) := \sum_{n=0}^{N} \sum_{j \in \Z} \bF^{n+1}(\rho_{j-1/2}^{n},\rho_{j+1/2}^{n}) \1_{\cP_{j+1/2}^{n}}.
\]

\begin{proposition}[Approximate entropy/constraint inequalities]
	(i) Fix $\varphi \in \Cc{\infty}(\R \times \R^{+}, \R^+)$ and $k \in [0,R]$. Then 
	there exists a constant $c_1$ depending only on $R, T, L$ and $\varphi$ such that 
	the following inequalities hold for all $0 \leq \tau < \tau' \leq T$:
	\begin{equation}
		\label{AEI}
		\begin{aligned}
            & \int_{\tau}^{\tau'} \int_{\R} |\rho_{\Delta}-k| \p_{t}\varphi + \Phi_{\Delta}\left( \rho_{\Delta},k \right) \p_{x}\varphi \; \d{x} \d{t} 
            + \int_{\R} |\rho_{\Delta}(x,\tau)-k| \varphi(x,\tau) \d{x} \\
			& - \int_{\R} |\rho_{\Delta}(x,\tau')-k| \varphi(x,\tau') \d{x} 
			+ 2 \int_{\tau}^{\tau'} \cR_{s_{\Delta}(t)}(k,q_{\Delta}(t)) \varphi(0,t) \d{t} \geq - c_1 (\Delta t + \Delta x).
		\end{aligned}
	\end{equation}
	
	(ii) Fix $\psi \in \Ck{\infty}([0,T], \R^+)$ and $\varphi \in \Cc{\infty}(\R, \R)$ 
	such that $\varphi(0)=1$. Then there exists a constant $c_2$ depending on 
	$R, T, L$, $Q$, $\varphi$ and $\psi$, such that for all 
	$0 \leq \tau < \tau' \leq T$:
	\begin{equation}
		\label{ACI}
		\begin{aligned}
			&  -\int_{\tau}^{\tau'} \int_{\R^{+}} \rho_{\Delta} \p_{t}(\varphi \psi) + \bF_{\Delta}(s_{\Delta},\rho_{\Delta}) \p_{x}(\varphi \psi) \; \d{x} \d{t}
			- \int_{\R^{+}} \rho_{\Delta}(x,\tau) \varphi(x) \psi(\tau) \d{x} \\
			& + \int_{\R^{+}} \rho_{\Delta}(x,\tau') \varphi(x) \psi(\tau') \d x \leq \int_{\tau}^{\tau'} q_{\Delta}(t) \psi(t) \d{t}
			+ c_2 (\Delta x + \Delta t).
		\end{aligned} 
	\end{equation}
\end{proposition}

\begin{proof}
	Fix $p,m \in \N$ such that $\tau \in [t^{p},t^{p+1} \mathclose[$ and 
	$\tau' \in [t^{m},t^{m+1} \mathclose[$.

	\textit{(i)} Define for all $n \in \N$ and $j \in \Z$, 
	$\ds{\varphi_{j+1/2}^{n} = \frac{1}{\Delta x \Delta t} \iint_{\cP_{j+1/2}^{n}} \varphi(x,t) \d x \d t}$. Multiplying the discrete entropy 
	inequalities \eqref{DEI} by $\varphi_{j+1/2}^{n}$, then summing over 
	$n \in [\![p; m-1]\!]$ and $j \in \Z$, one obtains after reorganization of the 
	sums (using in particular the Abel/“summation-by-parts” procedure)
	\begin{equation}
		\label{AEI_ineq}
		A + B + C + D + E  \geq 0,
	\end{equation}

	with
	\[
		\begin{aligned}
			A & = \sum_{n=p+1}^{m-1} \sum_{j \in \Z} |\rho_{j+1/2}^{n}-k| 
			\left(\varphi_{j+1/2}^{n}-\varphi_{j+1/2}^{n-1} \right) \Delta x, \quad 
			B = \sum_{n=p}^{m-1} \sum_{j \in \Z} \Phi_{j}^{n} 
			\left( \varphi_{j+1/2}^{n}-\varphi_{j-1/2}^{n} \right) \Delta t \\
			C & = \sum_{j \in \Z}|\rho_{j+1/2}^{p}-k| \varphi_{j+1/2}^{p} \Delta x 
			- \sum_{j \in \Z}|\rho_{j+1/2}^{m}-k| \varphi_{j+1/2}^{m-1} \Delta x \\
			D & = \sum_{n=p}^{m-1} \cR_{s^{n+1}}(k,q^{n+1}) \left( \varphi_{-1/2}^{n} + \varphi_{1/2}^{n}\right) \Delta t, \quad
			E = \sum_{n=p}^{m-1} \left( \Phi_{0}^{n} - \Phi_{int}^n \right) \left( \varphi_{-1/2}^{n} - \varphi_{1/2}^{n}\right) \Delta t.
		\end{aligned}
	\]

	Inequality \eqref{AEI} follows from \eqref{AEI_ineq} with
	\[
		\begin{split}
			c_1
            & = R \left(T \; \| \p_{tt}^{2} \varphi \|_{\L{\infty}([0, T], \L{1})} 
            + 4 \| \p_{t} \varphi \|_{\L{\infty}([0, T], \L{1})} \right) \\
			& + R L \left(T \; \| \p_{xx}^{2} \varphi \|_{\L{\infty}([0, T], \L{1})} 
            + 2 \| \p_{x} \varphi \|_{\L{\infty}([0, T], \L{1})} 
			+ 4 \| \varphi \|_{\L{\infty}} + 2 T \| \p_{x} \varphi \|_{\L{\infty}} \right),
		\end{split}
	\]

	making use of the bounds:
	\[
		\begin{split}
            & \left| A - \int_{\tau}^{\tau'} \int_{\R} |\rho_{\Delta}-k| \p_{t} \varphi \; \d x \d t \right| 
            \leq R \left(T \; \| \p_{tt}^{2} \varphi \|_{\L{\infty}([0, T], \L{1})} 
			+ 2 \| \p_{t} \varphi \|_{\L{\infty}([0, T], \L{1})} \right) \Delta t, \\
            & \left| B - \int_{\tau}^{\tau'} \int_{\R} \Phi_{\Delta}(\rho_{\Delta},k) \p_{x} \varphi \; \d x \d t \right| 
            \leq R L \left(T \; \| \p_{xx}^{2} \varphi \|_{\L{\infty}([0, T], \L{1})} \Delta x, 
			+ 2 \| \p_{x} \varphi \|_{\L{\infty}([0, T], \L{1})} \Delta t \right) \\
			& \left| C - \int_{\R} |\rho_{\Delta}(x,\tau) - k| \varphi(x,\tau) \d x 
            + \int_{\R} |\rho_{\Delta}(x,\tau') - k| \varphi(x,\tau') \d x \right| 
            \leq 2 R \; \| \p_{t} \varphi \|_{\L{\infty}([0, T], \L{1})} \Delta t, \\
			& \left| D - 2 \int_{\tau}^{\tau'} \cR_{s_{\Delta}(t)}(k,q_{\Delta}(t)) \varphi(0,t) \d t \right|
            \leq R L \left( 4 \| \varphi \|_{\L{\infty}} \Delta t + T\| \p_{x} \varphi \|_{\L{\infty}} \Delta x \right); \quad
            |E| \leq 2RT L \| \p_{x} \varphi \|_{\L{\infty}} \Delta x.
		\end{split}
	\]

	\textit{(ii)} In this case, the constant $c_2$ reads 
	\[
		\begin{aligned}
			c_2
			& = R \|\varphi\|_{\L{1}} \left( T \|\psi''\|_{\L{\infty}} + 4 \|\psi'\|_{\L{\infty}} \right) 
			+ \|Q\|_{\L{\infty}} \|\psi\|_{\L{\infty}} \left( 2 + T \|\varphi'\|_{\L{\infty}} \right) \\
			& + R L \|\psi\|_{\L{\infty}} \left( 2\|\varphi'\|_{\L{1}} + T \|\varphi'\|_{\L{1}} + T \|\varphi''\|_{\L{1}} \right).
		\end{aligned}
	\]

	Following the proof of \eqref{AEI}, define for all $n \in \N$ and $j \in \Z$,
	\[
		\psi^{n} = \frac{1}{\Delta t} \int_{t^{n}}^{t^{n+1}} \psi(t) \d t, \quad \text{and} \quad
		\varphi_{j+1/2} = \frac{1}{\Delta x} \int_{x_{j}}^{x_{j+1}} \varphi(x) \d x,
	\]

    multiply the scheme \eqref{MF} by $\varphi_{j+1/2} \psi^{n}$, then take the sum over $n \in [\![p; m-1]\!]$ and $j \geq 0$. Since the proof is very 
    similar to the one of \textit{(i)}, we omit the details.
\end{proof}

The final step is to obtain compactness for the sequences $(\rho_{\Delta})_{\Delta}$ and $(y_{\Delta})_{\Delta}$ in order to pass to the limit in 
\eqref{AEI}-\eqref{ACI}. We start with $(y_{\Delta})_{\Delta}$.

\begin{proposition}
	For all $t \in [0,T]$,
	\begin{equation}
		\label{AODE}
		y_{\Delta}(t) = y_o + \int_{0}^{t} \omega \left( \int_{\R} \rho_{\Delta}(x,u) \mu(x) \d{x} \right) \d{u}.
	\end{equation}

	Consequently, there exists $y \in \Czero([0,T], \R)$ such that up to an extraction, $(y_{\Delta})_{\Delta}$ converges uniformly to $y$ on $[0,T]$.
\end{proposition}

\begin{proof} 
	For all $t \in [0,T]$, if $t \in [t^{n},t^{n+1}\mathclose[$ for some 
	$n \in [\![0; N]\!]$, then we can write
	\[
		\begin{split}
			y_{\Delta}(t) - y_o
			& = \sum_{k=0}^{n-1} \int_{t^{k}}^{t^{k+1}} s^{k+1} \d u + \int_{t^{n}}^{t} s^{n+1} \d u \\
			& = \sum_{k=0}^{n-1} \int_{t^{k}}^{t^{k+1}} 
            \omega \left( \sum_{j \in \Z} \int_{\R} \rho^{k}_{j+1/2} \mu_{j+1/2} \Delta x \right) \d u 
            + \int_{t^{n}}^{t} \omega \left(\sum_{j \in \Z} \int_{\R} \rho^{n}_{j+1/2} \mu_{j+1/2} \Delta x \right) \d u \\
			& = \int_{0}^{t} \omega \left( \int_{\R} \rho_{\Delta}(x,u) \mu(x) \d x \right) \d u.
		\end{split}
	\]

	Let us also point out that from \eqref{speed_update}, we get that for all $\Delta$ and almost every $t \in \mathopen]0,T \mathclose[$,
	\begin{equation}
		\label{speed_VF}
		s_{\Delta}(t) = \omega \left( \int_{\R} \rho_{\Delta}(x,t)\mu(x) \d x \right).
	\end{equation}

	Combining \eqref{AODE} and \eqref{speed_VF}, we obtain that for all $\Delta$,
	\[
		\|\dot y_{\Delta}\|_{\L{\infty}} = \|s_{\Delta}\|_{\L{\infty}} \leq \|\omega\|_{\L{\infty}} 
		\quad \text{and} \quad \|y_{\Delta}\|_{\L{\infty}} \leq |y_o| + T \|\omega\|_{\L{\infty}}.
	\]

	The sequence $(y_{\Delta})_{\Delta}$ is therefore bounded in $\W{1}{\infty}(\mathopen]0,T \mathclose[, \R)$. Making use of the compact embedding of $\W{1}{\infty}(\mathopen]0,T \mathclose[, \R)$ in 
	$\Czero([0,T], \R)$, we get the existence of $y \in \Czero([0,T], \R)$ such that up to the extraction of subsequence, $(y_{\Delta})_{\Delta}$ converges uniformly to 
	$y$ on $[0,T]$.
\end{proof}

The presence of a time dependent flux in the conservation law of \eqref{AS2020} complicates the obtaining of compactness for $(\rho_{\Delta})_{\Delta}$. In 
particular, the techniques used in \cite{BGKT2008, BKT2009} to derive localized $\BV$ estimates don't apply here since our problem lacks time translation 
invariance. In the present situation, it would be possible to derive weak $\BV$ estimates (\cite{AGS2010, EGHBook}). We choose different options. Similarly to 
what we did in Section \ref{Section2}, we propose two ways to obtain compactness, which will lead to two convergence results.

\subsection{Compactness via one-sided Lipschitz condition technique}

First, we choose to adapt techniques and results put forward by Towers in \cite{TowersOSLC}. With this in mind, we suppose in this section that 
$f \in \Ck{2}([0,R], \R)$ satisfies:
\begin{equation}
	\label{ucc}
	\exists \alpha > 0, \ \forall \rho \in [0,R], \quad f''(\rho) \leq -\alpha.
\end{equation}

Though this assumption is stronger than the nondegeneracy one \eqref{non_deg}, since 
$f$ is bell-shaped, these two assumptions are similar in their spirit. We 
will also assume, following \cite{TowersOSLC}, that
\begin{equation}
    \label{flux}
    \text{the numerical flux chosen in} \ \eqref{MF} \ \text{is either the Engquist-Osher one or the Godunov one.}
\end{equation}

To be precise, the choice made for the numerical flux at the interface -- \textit{i.e.} when $j=0$ in \eqref{num_flux} -- does not play any role. What is important 
is that away from the interface, one chooses either the Engquist-Osher flux or the Godunov one. We denote for all $n \in [\![0; N+1]\!]$ and $j \in \Z$,
\[
	\bD_{j}^{n} := \max \left\{ \rho_{j-1/2}^{n} - \rho_{j+1/2}^{n},0 \right\}.
\]

We will also use the notation
\[
	\hat{\Z} := \Z \backslash \{-1,0,1\}.
\]

In \cite{TowersOSLC}, the author dealt with a discontinuous in both time and space flux and the specific "vanishing viscosity" coupling at the interface. The 
discontinuity in space was localized along the curve $\{x=0\}$. Here, we deal with only a discontinuous in time flux, but we also have a flux constraint along the 
curve $\{x=0\}$ since we work in the bus frame. The applicability of the technique of \cite{TowersOSLC} for our case with moving interface and flux-constrained 
interface coupling relies on the fact that one can derive a bound on $\bD_{j}^{n}$ as long as the "interface" does not enter the calculations for $\bD_{j}^{n}$ 
\textit{i.e.} $j \in \hat \Z$. This is what the following lemma points out under Assumptions \eqref{ucc}-\eqref{flux}. For readers' convenience and in order to 
highlight the generality of the technique of Towers \cite{TowersOSLC}, let us provide the key elements of the argumentation leading to compactness.

\begin{lemma}
	\label{OSLC_lemma1} 
	Let $n \in [\![0; N-1]\!]$ and $j \in \hat{\Z}$. Then setting 
	$\ds{a := \frac{\lambda \alpha}{4}}$, we have
	\begin{equation}
		\label{OSLC_bound1}
		\bD_{j}^{n+1} \leq 
		\max \left\{\bD_{j-1}^{n},\bD_{j}^{n},\bD_{j+1}^{n} \right\} - a \left( \max \left\{\bD_{j-1}^{n},\bD_{j}^{n},\bD_{j+1}^{n} \right\}  \right)^{2}
	\end{equation}

	and
	\begin{equation}
		\label{OSLC_bound2}
		\bD_{j}^{n+1} \leq \frac{1}{\min \{ |j|-1,n+1 \}a}.
	\end{equation}
\end{lemma}

\begin{proof}{\textit{(Sketched)}}
	Inequality \eqref{OSLC_bound2} is an immediate consequence of inequality \eqref{OSLC_bound1}, see \cite[Lemma 4.3]{TowersOSLC}. Obtaining inequality 
	\eqref{OSLC_bound1} however, is less immediate. Let us give some details of the proof. \\
	First, note that by introducing the function $\psi : z \mapsto z - az^{2}$, inequality \eqref{OSLC_bound1} can be stated as:
	\begin{equation}
		\label{OSLC_bound3}
		\bD_{j}^{n+1} \leq \psi \left( \max \left\{\bD_{j-1}^{n},\bD_{j}^{n},\bD_{j+1}^{n} \right\}\right).
	\end{equation}

	Then, one can show -- only using the monotonicity of both the scheme and the function $\psi$ -- that under the assumption
	\begin{equation}
		\label{OSLC_claim}
		\text{inequality} \ \eqref{OSLC_bound3} \ \text{holds when} \ (\rho_{j+3/2}^{n} - \rho_{j+1/2}^{n}), 
		\ (\rho_{j-1/2}^{n}-\rho_{j-3/2}^{n}) \leq 0,
	\end{equation}

	it follows that inequality \eqref{OSLC_bound3} holds for all cases. And finally in \cite[Page 23]{TowersOSLC}, the author proves that if the flux considered is 
	either the Engquist-Osher flux or the Godunov flux, then \eqref{OSLC_claim} holds. 
\end{proof}

The following lemma is an immediate consequence of inequality \eqref{OSLC_bound2}.

\begin{lemma}
	\label{OSLC_lemma2} 
	Fix $0 < \eps < X$. Let $i,J \in \N^{*}$ such that $\eps \in \bK_{i+1/2}$ and $X \in \bK_{J-1/2}$. Then if $\Delta x / \eps$ is sufficiently 
	small, there exists a constant $\ds{\bB= \bB\left(R,X,\frac{1}{a},\frac{1}{\eps}\right)}$, nondecreasing with respect to its arguments, such that for 
	all $n \geq i-1$,
	\begin{equation}
		\label{OSLC_bound4}
		\sum_{j=i}^{J-1} |\rho_{j+1/2}^{n}-\rho_{j-1/2}^{n}|, \sum_{j=-J+1}^{-i-1} |\rho_{j+1/2}^{n}-\rho_{j-1/2}^{n}| \leq \bB
	\end{equation}

	and
	\begin{equation}
		\label{OSLC_bound5}
		\sum_{j=i}^{J-2} |\rho_{j+1/2}^{n+1}-\rho_{j+1/2}^{n}|, \sum_{j=-J+1}^{-i-2} |\rho_{j+1/2}^{n+1}-\rho_{j+1/2}^{n}| \leq 2 \lambda L \bB.
	\end{equation}
\end{lemma}

\begin{proposition}
	There exists $\rho \in \L{\infty}(\R \times \mathopen]0,T \mathclose[, \R)$
	such that up to the extraction of a subsequence, $(\rho_{\Delta})_{\Delta}$ 
	converges almost everywhere to $\rho$ in $\R \times \mathopen]0,T \mathclose[$.
\end{proposition}

\begin{proof}
	Fix $0 < \eps < X$ and $t > \lambda \eps$. Denote by 
	$\Omega(X,\eps) := \mathopen]-X,-\eps\mathclose[\cup \mathopen]\eps, X\mathclose[$. 
	Introduce $i,J,n \in \N$ such that $\eps \in \bK_{i+1/2}$, $X \in \bK_{J-1/2}$ and 
	$t \in [t^{n},t^{n+1}\mathclose[$. Remark that
	\[
		(n+1) \Delta t > t > \lambda \eps \geq \lambda (i \ \Delta x) = i \Delta t,
	\]

	\textit{i.e.} $n \geq i-1$. Then, if we suppose that $\Delta x / \eps$ is 
	sufficiently small, we can use Lemma \ref{OSLC_lemma2}. From \eqref{OSLC_bound4}, 
	we get
	\begin{equation}
		\label{OSLC_bound6}
		\TV(\rho_{\Delta}(t)_{|\Omega(X,\eps)}) \leq 2 \bB
	\end{equation}

	and from \eqref{OSLC_bound5}, we deduce
	\begin{equation}
		\label{OSLC_bound7}
		\int_{\Omega(X,\eps)} |\rho_{\Delta}(x,t+\Delta t) - \rho_{\Delta}(x,t)| \d x \leq 4 L \bB \Delta t.
	\end{equation}

	Combining \eqref{OSLC_bound6}-\eqref{OSLC_bound7} and the $\L{\infty}$ bound \eqref{stab_scheme}, a functional analysis result (\cite[Theorem A.8]{HRBook}) 
	ensures the existence of a subsequence which converges almost everywhere to some $\rho$ on $\Omega(X,\eps) \times \mathopen]\lambda \eps,T\mathclose[$. 
	By a standard diagonal process we can extract a further subsequence (which we do 
	not relabel) such that $(\rho_{\Delta})_{\Delta}$ converges almost everywhere to 
	$\rho$ on $\R \times \mathopen]0,T \mathclose[$.
\end{proof}

\begin{theorem}
	\label{convergence_OSLC}
	Fix $\rho_o \in \L{1}(\R, [0,R])$ and $y_o \in \R$. Suppose that 
	$f \in \Ck{2}$ satisfies Assumptions \eqref{bell_shaped}-\eqref{bell_shaped_F}-\eqref{ucc}. Suppose also that in \eqref{num_flux}, we use the Engquist-Osher flux 
	or the Godunov one when $j \neq 0$ and any other monotone consistent and Lipschitz 
	numerical flux when $j=0$. Then under the CFL condition \eqref{CFL}, the scheme 
	\eqref{speed_update}~--~\eqref{num_flux} converges to an admissible weak solution to Problem \eqref{AS2020}.
\end{theorem}

\begin{proof}
	We have shown that -- up to the extraction of a subsequence -- $y_{\Delta}$ 
	converges uniformly on $[0,T]$ to some $y \in \Czero([0,T], \R)$ and that 
	$\rho_{\Delta}$ converges a.e. on $\R \times \mathopen]0,T \mathclose[$ to some 
	$\rho \in \L{\infty}(\R \times \mathopen]0,T \mathclose[, \R)$. We now prove that 
	this couple $(\rho,y)$ is an admissible weak solution to Problem \eqref{AS2020} in 
	the sense of Definition \ref{AWS}.

	Recall that for all $\Delta$ and $t \in [0,T]$,
	\[
		y_{\Delta}(t) = y_o + \int_{0}^{t} \omega \left( \int_{\R} \rho_{\Delta}(x,u) \mu(x) \d x \right) \d u.
	\]

	When letting $\Delta \to 0$, the Dominated Convergence Theorem ensures that $y$ satisfies \eqref{ODE}. Apply inequality \eqref{AEI} with 
	$\tau = 0$, $\tau' = T$, $\varphi \in \Cc{\infty}(\R^{*} \times [0, T\mathclose[, \R^+)$ 
    and $k \in [0,R]$ to obtain
	\[
		\int_{0}^{T} \int_{\R} |\rho_{\Delta}-k| \p_{t}\varphi + \Phi_{\Delta}(\rho_{\Delta},k) \p_{x} \varphi \; \d x \d t 
		+ \int_{\R}|\rho_{\Delta}^0-k|\varphi(x,0) \d x \geq 
        - c_1 (\Delta x + \Delta t).
	\]

	Then the a.e. convergence of $(s_{\Delta})_{\Delta}$ to $\dot y$ -- coming from \eqref{speed_VF} -- and the a.e. convergence of $(\rho_{\Delta})_{\Delta}$ to 
	$\rho$ ensure that when letting $\Delta \to 0$, we get
	\[
		\int_{0}^{T} \int_{\R} |\rho-k| \p_{t}\varphi + \Phi_{\dot y(t)}(\rho,k) \p_{x} \varphi \d x \d t 
		+ \int_{\R}|\rho_o(x+y_o)-k|\varphi(x,0) \d x \geq 0,
	\]

	and consequently $\rho \in \Czero([0,T], \Lloc{1}(\R, \R))$, see Remark \ref{Rk2}. 
	Now, we pass to the limit in \eqref{AEI} and \eqref{ACI} using the 
	a.e. convergence of $(s_{\Delta})_{\Delta}$ to $\dot y$ and of $(\rho_{\Delta})_
	{\Delta}$ to $\rho$ as well as the continuity of $Q$ and $\omega$. 
	We find that for all test functions $\varphi \in \Cc{\infty}(\R\times\R^{+}, \R^+)$ 
	and $k \in [0,R]$, the following inequalities hold for almost every 
	$0 \leq \tau < \tau' \leq T$:
	\[
		\begin{aligned}
			& \int_{\tau}^{\tau'} \int_{\R} |\rho-k| \p_{t}\varphi + \Phi_{\dot y(t)}(\rho,k) \p_{x} \varphi \; \d x \d t 
			+ \int_{\R}|\rho(x,\tau)-k|\varphi(x,\tau) \d x \\
			& - \int_{\R}|\rho(x,\tau')-k|\varphi(x,\tau') \d x  + 2 \int_{\tau}^{\tau'} \cR_{\dot y(t)}(k,q(t)) \varphi(0,t) \d t \geq 0.
		\end{aligned}
	\]

	To conclude, note that the expression in the left-hand side of the previous inequality is a continuous function of $(\tau,\tau')$ which is almost everywhere 
	greater than the continuous function $0$. By continuity, this expression is everywhere greater than $0$, which proves that $\rho$ satisfies the entropy 
    inequalities \eqref{EI}. Using similar arguments, one shows that $\rho$ also satisfies the constraint inequalities \eqref{CI}. This shows that the couple 
    $(\rho,y)$ is an admissible weak solution to \eqref{AS2020}, and that concludes the proof of convergence.
\end{proof}

We proved than in the $\L{\infty}$ framework, the scheme converges to an admissible weak solution, but note that there is no guarantee of uniqueness in this 
construction. Also stress that we cannot extend this result to general consistent monotone numerical fluxes beyond hypothesis \eqref{flux}.

\subsection{Compactness via global BV bounds}

The following result is the discrete version of Lemma \ref{bus_BV}, so it is consistent 
that the proof uses the discrete analogous arguments of the ones we used in 
the proof of Lemma \ref{bus_BV}.

\begin{lemma}
	\label{bus_BV_VF}
	Introduce for all $\Delta > 0$ the function $\xi_{\Delta}$ defined for all $t \in [0,T]$ by
	\[
		\xi_{\Delta}(t) := \int_{\R} \rho_{\Delta}(x,t) \mu(x) \d{x}.
	\]

	Then $\xi_{\Delta}$ has bounded variation and consequently, so does $s_{\Delta}$.
\end{lemma}

\begin{proof}
	Since $\mu \in \BV(\R, \R)$, there exists a sequence of smooth functions $(\mu_{\ell})_{\ell \in \N} \subset \BV \cap \Cc{\infty}(\R, \R)$ such that
	\[
		\| \mu_{\ell} - \mu \|_{\L{1}} \limit{\ell}{+\infty} 0 \quad \text{and} \quad \TV(\mu_{\ell}) \limit{\ell}{+\infty} \TV(\mu).
	\]

	Introduce for all $\ell \in \N$ and $t \in [0,T]$, the function $\ds{\xi_{\Delta, \ell}(t) = \int_{\R} \rho_{\Delta}(x,t) \mu_{\ell}(x) \d x}$ and let 
	$K > 0$ such that
	\[
		\forall \ell \in \N, \quad \|\mu_{\ell}\|_{\L{1}}, \TV(\mu_{\ell}) \leq K.
	\]

	For all $\ell \in \N$ and $t,s \in [0,T]$, if $t \in [t^{k},t^{k+1}\mathclose[$ and $s \in [t^{m},t^{m+1}\mathclose[$, we have
	\[
		\begin{split}
			\left| \xi_{\Delta, \ell}(t) - \xi_{\Delta, \ell}(s) \right|
			& = \left| \xi_{\Delta, \ell}(t^{k}) - \xi_{\Delta, \ell}(t^{m}) \right| \\
			& = \left| \int_{\R} \rho_{\Delta}(x,t^{k}) \mu_{\ell}(x) \d x - \int_{\R} \rho_{\Delta}(x,t^{m}) \mu_{\ell}(x) \d x \right| \\
            & = \left| \sum_{j \in \Z} (\rho^{k}_{j+1/2} - \rho^{m}_{j+1/2}) \mu^{\ell}_{j+1/2} \Delta x \right|, 
            \quad \mu^{\ell}_{j+1/2} = \frac{1}{\Delta x} \int_{x_j}^{x_{j+1}} \mu_{\ell}(x) \d{x} \\
			& = \left| \sum_{j \in \Z} \sum_{\tau=m}^{k-1} (\rho^{\tau+1}_{j+1/2} - \rho^{\tau}_{j+1/2}) \mu^{\ell}_{j+1/2} \Delta x \right| \\
			& = \left| \sum_{\tau=m}^{k-1} \sum_{j \in \Z} \left( \bF_{j}^{\tau+1}(\rho_{j-1/2}^{\tau},\rho_{j+1/2}^{\tau}) - 
			\bF_{j+1}^{\tau+1}(\rho_{j+1/2}^{\tau},\rho_{j+3/2}^{\tau}) \right) \mu^{\ell}_{j+1/2} \Delta t \right| \\
			 &= \left| \sum_{\tau=m}^{k-1} \sum_{j \in \Z} \bF_{j+1}^{\tau+1}(\rho_{j+1/2}^{\tau},\rho_{j+3/2}^{\tau})
			(\mu_{j+3/2}^{\ell} - \mu_{j+1/2}^{\ell}) \Delta t \right| \\
            & \leq R L \sum_{\tau=m}^{k-1} \TV(\mu_{\ell}) \Delta t \leq R LK (|t-s| + 2\Delta t).
		\end{split}
	\]

	Consequently, for all $\ell \in \N$, $\Delta > 0$ and $t,\tau \in [0,T]$, the triangle inequality yields:
	\[
		\left|\xi_{\Delta}(t) - \xi_{\Delta}(\tau) \right| \leq 2R \|\mu-\mu_{\ell}\|_{\L{1}} + R LK (|t-\tau| + 2\Delta t).
	\]

	Letting $\ell \to +\infty$, we get that for all $\Delta > 0$ and $t,\tau \in [0,T]$,
	\[
		\left|\xi_{\Delta}(t) - \xi_{\Delta}(\tau) \right| \leq R LK (|t-\tau| + 2\Delta t),
	\]

	which leads to
	\[
		\TV(\xi_{\Delta}) = \sum_{k=0}^{N} \left|\xi_{\Delta}(t^{k+1}) - \xi_{\Delta}(t^{k}) \right| \leq 3 R LK (T + \Delta t).
	\]

	This proves that $\xi_{\Delta} \in \BV([0,T], \R)$. Since $\omega$ is Lipschitz continuous, $s_{\Delta}$ also has bounded variation.
\end{proof}

\begin{theorem}
	\label{convergence_BV}
	Fix $\rho_o \in \L{1} \cap \BV(\R, [0,R])$ and $y_o \in \R$. Suppose that $f$ satisfies \eqref{bell_shaped}-\eqref{bell_shaped_F}-\eqref{reg_flux} and 
	that $Q$ satisfies \eqref{level_constraint_splitting}. Suppose also that in \eqref{num_flux}, we use the Godunov flux when $j = 0$ and any other monotone 
	consistent and Lipschitz numerical flux when $j \neq 0$. Then under the CFL condition \eqref{CFL}, the scheme \eqref{speed_update}~--~\eqref{num_flux} converges 
	to a $\BV$-regular solution to Problem \eqref{AS2020}.
\end{theorem}

\begin{proof}
	All the hypotheses of Lemma \ref{compactness_density_appendix} are fulfilled. Consequently, there exists a constant $C_\eps > 0$ such that for all 
	$n \in [\![0; N-1]\!]$,
	\begin{equation}
		\label{convergence_BV_bound1}
		\begin{aligned}
			\TV\left( \rho_{\Delta} (t^{n+1}) \right) 
			& \leq \TV(\rho_o) + 4R + C_\eps \left( \sum_{k=0}^{n} \left|q^{k+1}- q^{k} \right| 
			+ \sum_{k=0}^{n} \left|s^{k+1}- s^{k} \right| \right) \\
			& \leq \TV(\rho_o) + 4R + C_\eps (1+\|Q'\|_{\L{\infty}}) \sum_{k=0}^{n} \left|s^{k+1}- s^{k} \right|.
		\end{aligned}
	\end{equation}

	Making use of Lemma \ref{bus_BV_VF}, we obtain that for all $n \in [\![0; N]\!]$,
	\[
        \sum_{k=0}^{n} |s^{k+1}- s^{k}|
        = \sum_{k=0}^{n} |s_{\Delta}(t^{k+1})- s_{\Delta}(t^{k})|
        \leq \|\omega\|_{\L{\infty}} \sum_{k=0}^{n} |\xi_{\Delta}(t^{k+1})- \xi_{\Delta}(t^{k})|
        \leq 3 R LK \|\omega\|_{\L{\infty}} (T + \Delta t).
	\]
	
	where the constant $K$ was introduced in the proof of Lemma \ref{bus_BV_VF}. The two last inequalities imply that for all $t \in [0,T]$, we have
	\begin{equation}
		\label{convergence_BV_bound2}
		\TV(\rho_{\Delta}(t)) \leq \TV(\rho_o) + 4R + 3C_\eps (1+\|Q'\|_{\L{\infty}}) \|\omega\|_{\L{\infty}} R L K (T + \Delta t).
	\end{equation}

	Therefore, the sequence $(\rho_{\Delta})_{\Delta}$ is uniformly in time bounded in $\BV(\R, \R)$. Using \cite[Appendix]{DE2016}, we get the existence of 
	$\rho \in \Czero([0,T], \Lloc{1}(\R, \R))$ such that
	\[
		\forall t \in [0,T], \quad 
		\rho_{\Delta}(t) \limit{\Delta}{0} \rho(t) \; \text{in} \; \Lloc{1}(\R, \R).
	\]

	Following the proof of Theorem \ref{convergence_OSLC}, we show that $(\rho,y)$ is an admissible weak solution. Then passing to the limit in 
	\eqref{convergence_BV_bound2}, the lower semi-continuity of the $\BV$ semi-norm ensures that $(\rho,y)$ is also $\BV$-regular.
\end{proof}

\begin{remark}
	Note the complementarity of the hypotheses made in the above theorem with the ones of Theorem \ref{convergence_OSLC}. Recall that in Theorem 
	\ref{convergence_OSLC}, we needed the Godunov flux only away from the interface.
\end{remark}

\newpage 

\section{Numerical simulations}
\label{Section4}

In this section we present some numerical tests performed with the scheme analyzed in Section \ref{Section3}. In all the simulations we take the uniformly concave 
flux $f(\rho) = \rho(1-\rho)$ (the maximal car velocity and the maximal density are assumed to be equal to one). Following the hypotheses of Theorem 
\ref{convergence_BV}, we choose the Godunov flux at the interface, and the Rusanov one away from the interface. We will use weight functions of the kind
\[
	\mu_k (x) = 2^k \1_{\left[ 0 ; \frac{1}{2^k} \right]}(x),
\]

for one (in Section \ref{NS1}) or several (in Section \ref{NS2}) values of $k \in \N^*$.

\subsection{Validation of the scheme}
\label{NS1}

In this section, consider a two-lane road on which a bus travels with a speed given by the function
\[
	\omega (\rho) = 
	\left\{
		\begin{array}{ccc}
			\ds{\frac{\alpha}{(\beta + \rho)^2}} & \text{if} & 0 \leq \rho \leq 0.6 \\[10pt]
			1 - \rho & \text{if} & 0.6 \leq \rho \leq 1,
		\end{array}
	\right.
\]

where $\alpha$ and $\beta$ are chosen so that $\omega(0) = 0.7$ and $\omega(0.6) = 0.4$, as illustrated in Figure \ref{fig1} (left). The set-up of the 
experiment is the following. Consider a domain of computation $[0, 11]$, the weight function $\mu_4$ and the following data:
\[
	\rho_o(x) = 0.5 \1_{[0.5 ; 1]} (x), \quad y_o = 1.5, \quad Q(s) = 0.75 \times \left( \frac{1-s}{2} \right)^2.
\]

The idea behind the choice of $Q$ is that in average (between the two lanes), the presence of the slow vehicle reduces by $25\%$ the maximum traffic flow. As 
we can see in Figure \ref{fig1} (right), the slow vehicle nearly always travels at maximum velocity. It makes sense because even though we can see that cars are 
overtaking it (Figure \ref{fig1}, right and Figure \ref{fig2}), the density $\xi$ ahead of it is never sufficiently important to make it go slower. 

\begin{figure}[!htp]
	\begin{center}
	  \includegraphics[scale = 0.60]{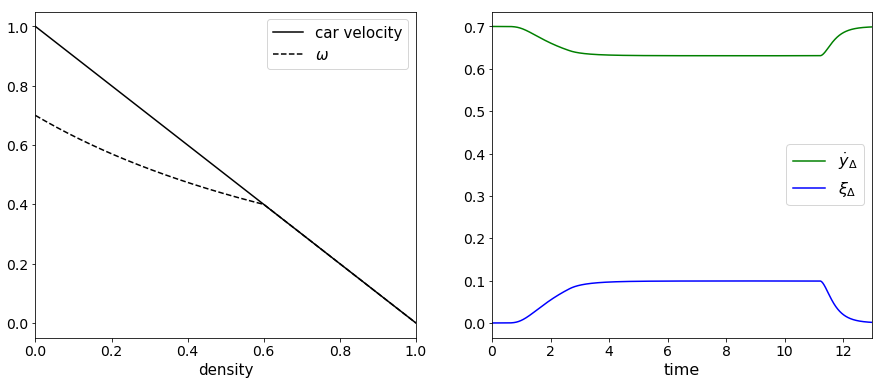}
    \end{center}
    \vspace*{-5mm}
    \caption{Evolution in time of the bus velocity $\dot y_\Delta$ and of the subjective density $\xi_\Delta$, with $\Delta x = 0.01$.}
    \label{fig1}
\end{figure}

\newpage

\begin{figure}[!htp]
	\begin{center}
	  \includegraphics[scale = 0.60]{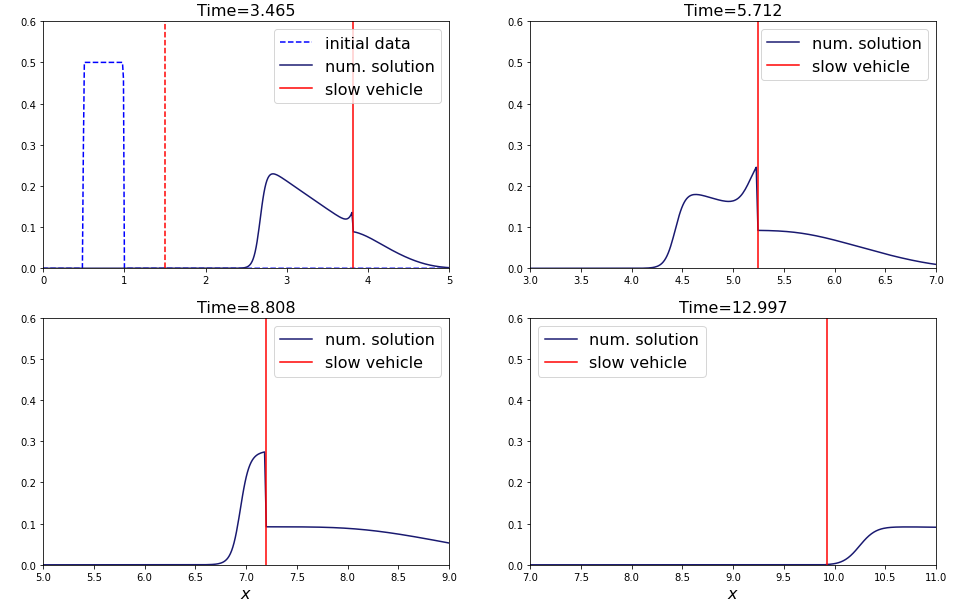}
    \end{center}
    \vspace*{-5mm}
    \caption{The numerical solution at different fixed times, red dashed lines correspond to the slow vehicle initial position.}
    \label{fig2}
\end{figure}

\begin{remark}
	The function $\omega$ we chose above is not of the form as required in \cite{LP2018, LP2020}. Once again, let us stress that the particular form 
	$\omega(\rho) = \min\left\{V_{\text{bus}}, 1-\rho \right\}$, where $V_{\text{bus}}$ is the maximum bus velocity, is crucial for the well-posedness result 
	of \cite{LP2018, LP2020} to hold. Indeed, it is essential in the analysis of \cite{LP2018, LP2020} that the velocity of the bus be constant (equal to 
	$V_{\text{bus}})$ across the non-classical shocks. Our non-local model is not bound to this restriction.
\end{remark}

\subsection{Convergence analysis}

We also perform a convergence analysis for this test. In the Table \ref{table1}, we computed the relative errors
\[
	\mathbf{E}_{\rho, \Delta} := \| \rho_{\Delta} - \rho_{\Delta/2} \|_{\L{1}(\mathopen]0,T \mathclose[, \L{1}(\R))} 
	\quad \text{and} \quad 
	\mathbf{E}_{y, \Delta} := \|y_{\Delta} - y_{\Delta/2} \|_{\L{\infty}([0, T])},
\]

for different number of space cells at the final time $T = 13$. We see (Figure \ref{fig3}) that those ratios converge with convergence orders approximately 
equal to $0.76$ for the car density and approximately equal to $1.1$ for the slow moving vehicle position.

\newpage

\begin{table}[!htp]
    \begin{minipage}[b]{0.4\linewidth}
    \centering
        \begin{tabular}{ccc}
            \hline
            Number of cells & $\mathbf{E}_{\rho, \Delta} \ (\times 10^{-2})$ & $\mathbf{E}_{y, \Delta} \ (\times 10^{-3})$ \\ 
            \hline\hline
            160 & $24.053$ & $48.0643$ \\
            320 & $15.731$ & $15.939$ \\
            640	& $9.647$ & $7.698$ \\
            1280 & $6.197$ & $3.715$ \\ 
            2560 & $3.226$ & $1.777$ \\
            5120 & $1.936$ & $0.889$ \\
            10240 & $1.055$ & $0.443$ \\
            \hline
        \end{tabular}
        \caption{Measured errors ($T = 13$).}
        \label{table1}
    \end{minipage}\hfill
    \begin{minipage}[b]{0.45\linewidth}
    \centering
        \includegraphics[scale=0.55]{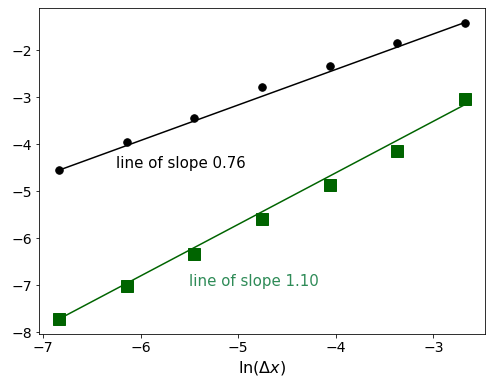}
        \captionof{figure}{Rates of convergence for $\rho_\Delta$ (in black) and $y_\Delta$ (in green), with $T = 13$.}
        \label{fig3}
    \end{minipage}
\end{table}

\subsection{Comparisons with experiments on the local model}
\label{NS2}

Now we confront the numerical tests performed with our model with the tests done by the authors in \cite{CDMG2017} approximating the original problem of 
\cite{DMG2014}. We deal with a road of length $1$ parametrized by the interval $[0,1]$ and choose the weight function $\mu_3$. Moreover,
\[
	\omega(\rho) = \min\{0.3 ; \ 1 - \rho \} \quad \text{and} \quad Q(s) = 0.6 \times \left( \frac{1-s}{2} \right)^2.
\]

First, consider the initial datum
\begin{equation}
	\label{case1}
	\rho_o(x) = 
	\left\{ 
		\begin{array}{ccc} 
			0.4 & \text{if} & x < 0.5 \\ 
			0.5 & \text{if} & x > 0.5 
		\end{array} 
	\right. \ y_o = 0.5.
\end{equation}

The numerical solution is composed of two classical shocks separated by a non-classical discontinuity, as illustrated in Figure \ref{fig4} (left). Next, we choose
\begin{equation}
	\label{case2}
	\rho_o(x) = 
	\left\{ 
		\begin{array}{ccc} 
			0.8 & \text{if} & x < 0.5 \\ 
			0.5 & \text{if} & x > 0.5 
		\end{array}
	\right. \ y_o = 0.5.
\end{equation}

The values of the initial condition create a rarefaction wave followed by a non-classical and classical shocks, as illustrated in Figure \ref{fig4} (right).

\newpage

\begin{figure}[!htp]
	\begin{center}
		\includegraphics[scale = 0.60]{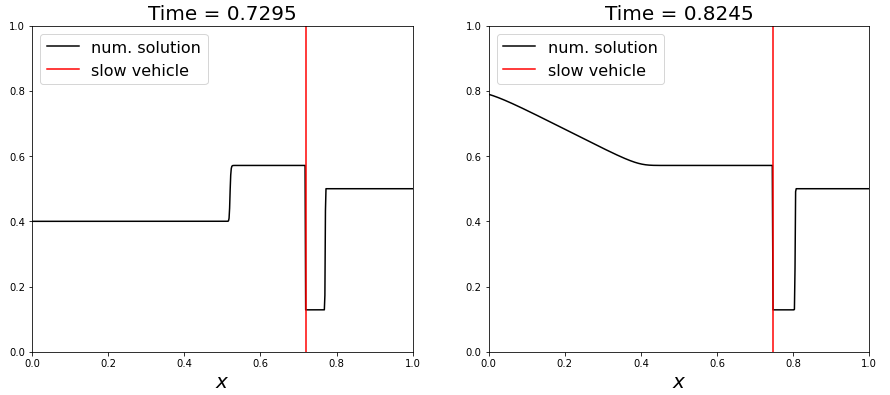}
	\end{center}
	\vspace*{-5mm}
	\caption{Evolution in time of the numerical density corresponding to initial data \eqref{case1} (left) and \eqref{case2} (right), with $\Delta x = 0.001$.}
	\label{fig4}
\end{figure}

Finally, still following \cite{CDMG2017}, we consider
\begin{equation}
	\label{case3}
	\rho_o(x) = 
	\left\{ 
		\begin{array}{ccc} 
			0.8 & \text{if} & x < 0.5 \\ 
			0.4 & \text{if} & x > 0.5 
		\end{array}
	\right. \ y_o = 0.4.
\end{equation}

Here the solution is composed of a rarefaction wave followed by non-classical and classical shocks on the density that are created when the slow vehicle
approaches the rarefaction and initiates a moving bottleneck, as illustrated in Figure \ref{fig5}. 

\begin{figure}[!htp]
	\begin{center}
		\includegraphics[scale = 0.60]{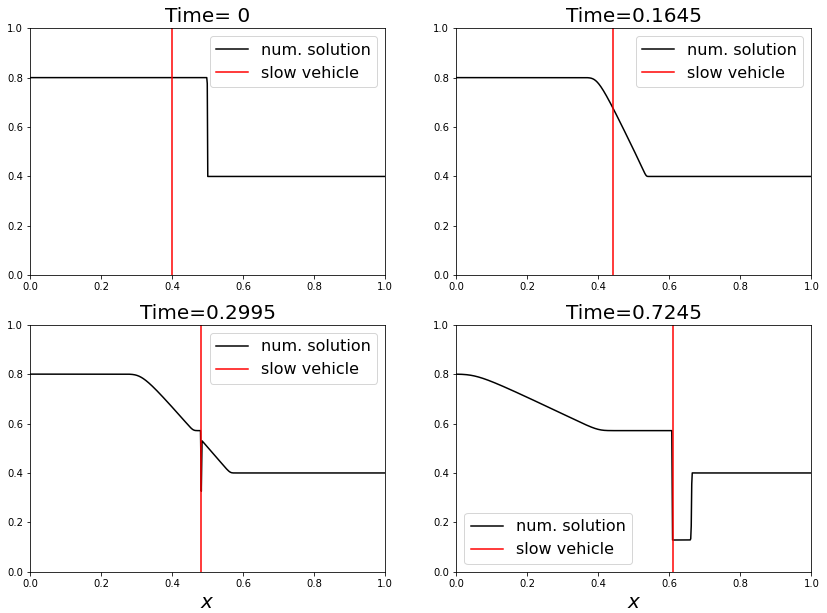}
	\end{center}
	\vspace*{-5mm}
	\caption{Evolution in time of the numerical density corresponding to initial data \eqref{case3}, with $\Delta x = 0.001$.}
	\label{fig5}
\end{figure}

With these three tests, we can already see -- in a qualitative way -- the resemblance between the numerical approximations to the solutions to our model and the 
numerical approximations of \cite{CDMG2017}. One way to quantify their proximity is for example to evaluate the $\L{1}$ error between the car densities and the 
$\L{\infty}$ error between the bus positions. More precisely, denote by $(\rho_{\Delta},y_{\Delta})$ the approximation of the $\BV$-regular solution to 
\eqref{AS2020} obtained with the scheme \eqref{speed_update}~--~\eqref{num_flux}, and denote by $(\overline{\rho}_{\Delta},\overline{y}_{\Delta})$ the couple 
obtained with this same scheme but
\[
	\text{replacing} \quad 
	s^{n+1} = \omega \left( \sum_{j \in \Z} \rho_{j+1/2}^{n} \mu_{j+1/2} \Delta x \right) \quad \text{by} \quad 
	s^{n+1} = \omega \left(\rho_{1/2}^{n} \right).
\]

Let us precise that this is not the scheme the authors of \cite{CDMG2017} proposed. However, this scheme is consistent with the problem
\begin{equation}
	\label{DMG}
	\left\{
		\begin{array}{lcr}
			\p_{t}\rho + \p_{x} \left(F(\dot y(t),\rho) \right) = 0 & & \R \times \mathopen]0,T \mathclose[ \\[5pt]
			\rho(x,0) = \rho_o(x+y_o) & & x \in \R \\[5pt]
			\left. F(\dot y(t),\rho) \right|_{x=0} \leq Q(\dot y(t)) & & t \in \mathopen]0,T \mathclose[ \\[5pt]
			\ds{\dot y(t) = \omega \left( \rho(0+,t) \right)} & & t \in \mathopen]0,T \mathclose[ \\[5pt]
			y(0) = y_o & & 
		\end{array}
	\right.
\end{equation}

and behaves in a stable way in the calculations we performed. Therefore, the couple 
$(\overline{\rho}_{\Delta},\overline{y}_{\Delta})$ is expected to give a 
reasonable approximation of the solution to \eqref{DMG}. With this in mind, for the case \eqref{case3} and still with the weight function $\mu_3$, we 
computed in Table \ref{table2} the measured errors 
\[
	\mathbf{E}^1_\Delta := \|\rho_{\Delta} - \overline{\rho_\Delta} \|_{\L{1}(\mathopen]0,T \mathclose[, \L{1}(\R))} \quad \text{and} \quad
	\mathbf{E}^{\infty}_\Delta :=  \|y_{\Delta} - \overline{y_\Delta} \|_{\L{\infty}([0, T])}.
\]

\begin{table}[!htp]
	\centering
	\begin{tabular}{ccc}
		\hline
		Number of cells & $\mathbf{E}^1_\Delta \ (\times 10^{-4})$ & $\mathbf{E}^{\infty}_\Delta \ (\times 10^{-3})$ \\ 
		\hline\hline
		160		& $32.672$ & $18.519$ \\
		320 	& $14.236$ & $7.341$ \\
		640		& $5.837$ & $3.701$ \\
		1280 	& $3.833$ & $4.879$ \\
		2560 	& $3.207$ & $6.405$ \\
		5120 	& $2.922$ & $7.144$ \\
		10240 & $2.776$ & $7.501$ \\
		20480 & $2.698$ & $7.674$ \\
		40960 & $2.658$ & $7.759$ \\
		\hline
	\end{tabular}
	\caption{Measured errors at time $T=0.7245$.}
	\label{table2}
\end{table}

These calculations indicate that for a sufficiently large number of cells $J \geq 40960$,
\[
	\mathbf{E}^1_\Delta \simeq 2.7 \times 10^{-4} \quad \text{and} \quad \mathbf{E}^{\infty}_\Delta \simeq 7.6 \times 10^{-3}.
\]

This indicates the discrepancy between our non-local and the local model \eqref{DMG} of \cite{DMG2014}. The idea is now to fix the number of cells $J = 40960$ and 
to make the length of the weight function support go to zero. In Table \ref{table3}, we have computed, for different weight functions, the error between the 
approximations of the two models. This error corresponds, as in the above calculation, to the residual error observed starting from a sufficiently small $\Delta x$.

\begin{table}[!htp]
	\centering
		\begin{tabular}{cccc}
			\hline
			weight function & $\mathbf{E}^1_\Delta$ & $\mathbf{E}^{\infty}_\Delta$ \\ 
			\hline\hline
			$\mu_1$ & $6.810 \times 10^{-3}$ & $5.489 \times 10^{-2}$ \\
			$\mu_2$ & $1.105 \times 10^{-3}$ & $1.972 \times 10^{-2}$ \\
			$\mu_3$ & $2.658 \times 10^{-4}$ & $7.759 \times 10^{-3}$ \\
			$\mu_4$ & $9.232 \times 10^{-5}$ & $2.913 \times 10^{-3}$ \\
			$\mu_5$ & $6.190 \times 10^{-5}$ & $9.110 \times 10^{-4}$ \\
			\hline  
		\end{tabular}
	\caption{Measured errors at time $T=0.7245$}
	\label{table3}
\end{table}

\begin{remark}
	\label{rk_criticism}
	The previous simulations show a closeness between our model as $\mu \to \delta_{0^{+}}$ and \eqref{DMG}. Let us however point that the non-locality in 
	space for the slow vehicle introduces an undesirable artifact into the model. In the rarefaction regime one may observe that the large vehicle may move a bit 
	faster that the surrounding flow. The situation where this effect becomes truly perceptible is when considering initial data of the type
	\begin{equation}
		\label{rev2}
		\rho_o(x) = 
		\left\{
			\begin{array}{ccl}
				1 & \text{if} & x < x_{b} \\[5pt]
				0 & \text{if} & x > x_{b}.
			\end{array}
		\right.
	\end{equation}

	Indeed, for such data, there exists a small time interval $[0, \delta]$ in which $\dot y(t) > v(\rho(y(t)^+)) = 0$, which would suggest that the slow 
	vehicle moves forward while the cars in front of it do not. This time interval is in fact quite small due to the narrowness of the support of the 
	weight function. The local model does not develop such phenomena. This qualitative artifact precludes us from giving a microscopic interpretation to the model, 
	which main output is the global influence of the slow vehicle on the flow; however, let us stress that the phenomenon becomes quantitatively negligible for 
	larger times. Indeed, Ole\u{ı}nik estimate on decay of positive waves ensures that data of the type \eqref{rev2} evolve into 
	rarefaction waves and do not appear while driving: the classical LWR model precludes the formation of rarefaction waves focused at positive time. The 
	modification of the classical LWR brought by the constraint may produce non-classical waves at positive times; while these waves are downward jumps in density 
	like in \eqref{rev2}, they are situated precisely at the location of the constraint and not slightly behind it, like in \eqref{rev2}.
\end{remark}

Even if we are unable, at this time, to rigorously link our problem \eqref{AS2020} with $\mu \to \delta_{0^{+}}$ and the original problem \eqref{DMG} of the 
authors in \cite{DMG2014}, this last experiment corroborates the conjecture that the local model \eqref{DMG} is the singular limit of our model in the case 
$\omega$ is of the form $\omega(\rho) = \min\left\{V_{\text{bus}} ; 1-\rho \right\}$. The other interesting question is whether the local model is well posed 
beyond this particular choice of $\omega$.

\begin{thx}
    The author is most grateful to Boris Andreianov for his constant support and many enlightening discussions.
\end{thx}

\newpage 

\appendix

\section{On BV bounds for limited flux models}

We focus on the study of the following class of models:
\begin{equation}
	\label{AS_appendix}
	\left\{
		\begin{array}{lcr}
			\p_{t}\rho + \p_{x} \left(F(s(t),\rho) \right) = 0 & &\R \times \mathopen]0,T \mathclose[ \\[5pt]
			\rho(x,0) = \rho_o(x) & & x \in \R \\[5pt]
			\left. F(s(t),\rho) \right|_{x=0} \leq q(t) & & t \in \mathopen]0,T \mathclose[,
		\end{array}
	\right.
\end{equation}

where $s \in \BV([0,T], [0,\Sigma])$ for some $\Sigma > 0$ and $q \in \BV([0,T], \R^{+})$. We suppose that $F \in \Ck{1}([0, \Sigma] \times [0,R], \R)$ and that for 
all $s \in [0,\Sigma]$, $\rho \mapsto F(s, \rho)$ is bell-shaped:
\begin{equation}
	\label{bell_shaped_appendix}
		F(s,0) = 0, \; F(s,R) \leq 0 \; \text{and} \;
		\exists! \; \overline{\rho}_{s} \in \mathopen]0,R \mathclose[, \; \p_{\rho}F(s,\rho) \left( \overline{\rho}_{s} - \rho \right) > 0 \; \text{for a.e.} \; \rho \in \mathopen]0,R \mathclose[.
\end{equation}

This framework covers the particular case when $F$ takes the form:
\[
	F(s(t),\rho) = f(\rho) - s(t) \rho,
\]

with bell-shaped $f : [0,R] \to \R^+$, which our model \eqref{AS2020} is based on. This 
class of models is well known, especially when the flux function is not 
time dependent, \textit{cf.} \cite{CG2007, AGS2010}. In this appendix, we establish in 
passing the well-posedness of Problem \eqref{AS_appendix}, but our main 
interest lies in the $\BV$ in space regularity of the solutions. More precisely, we aim 
at obtaining a bound on the total variation of the solutions to 
\eqref{AS_appendix}, using a finite volume approximation which allows for sharp control 
of the variation at the constraint. Note that the alternative offered by 
wave-front tracking would be cumbersome because of the explicit time-dependency in 
\eqref{AS_appendix}. In the general case, entropy solutions to limited flux 
problems like \eqref{AS_appendix} do not belong to 
$\L{\infty}(\mathopen]0,T \mathclose[, \BV(\R))$, see \cite{ADGV2011}. We will show 
that it is the case under a mild assumption on 
the constraint function $q$ -- see Assumption \eqref{level_constraint_appendix} below --  and provided that
\[
	\rho_o \in \L{1} \cap \BV (\R, [0,R]).
\]

Throughout the appendix, for all $s \in [0,\Sigma]$ and $a,b \in [0,R]$, we denote by
\[
	\Phi_{s}(a,b) = \text{sign}(a-b)(F(s,a) - F(s,b))
\]

the classical Kru{\v{z}}kov entropy flux associated with the Kru{\v{z}}kov entropy $\rho \mapsto |\rho-k|$, for all $k \in [0,R]$, see \cite{Kruzhkov1970}.

\subsection{Equivalent definitions of solution and uniqueness}

Let us first recall the following definition.

\begin{definition}
	\label{AWS_appendix}
	We say that $\rho \in \L{\infty}(\R \times \mathopen]0,T \mathclose[, \R)$ is an admissible weak solution to \eqref{AS_appendix} if

	(i) the following regularity is fulfilled: $\ds{\rho \in \Czero([0,T], \Lloc{1}(\R, \R))}$;

	(ii) for all test functions $\varphi \in \Cc{\infty}(\R \times \R^{+}, \R^+)$ and 
	$k \in [0,R]$, the following entropy inequalities are verified for all 
	$0 \leq \tau < \tau' \leq T$:
	\[
		\begin{aligned}
			& \int_{\tau}^{\tau'} \int_{\R} |\rho-k| \p_{t}\varphi + \Phi_{s(t)}(\rho,k) \p_{x} \varphi \; \d x \d t 
			+ \int_{\R} |\rho(x,\tau)-k|\varphi(x,\tau) \d x \\
			& - \int_{\R}|\rho(x,\tau')-k|\varphi(x,\tau') \d x + 2 \int_{\tau}^{\tau'} \cR_{s(t)}(k,q(t)) \varphi(0,t) \d t \geq 0,
		\end{aligned}
	\]

	where
	\[
		\cR_{s(t)}(k,q(t)) :=  F(s(t),k) - \min\left\{ F(s(t),k),q(t) \right\};
	\]

	(iii) for all test functions $\psi \in \Cc{\infty}([0,T], \R^+)$ and some given $\varphi \in \Cc{\infty}(\R, \R)$ which verifies $\varphi(0)=1$, the 
	following weak constraint inequalities are verified for all $0 \leq \tau < \tau' \leq T$:
	\[
		\begin{aligned}
			& - \int_{\tau}^{\tau'} \int_{\R^{+}} \rho \p_{t} (\varphi \psi) + F(s(t),\rho) \p_{x}(\varphi \psi) \ \d x \d t 
			- \int_{\R^{+}} \rho(x,\tau) \varphi(x) \psi(\tau) \d x \\
			& + \int_{\R^{+}} \rho(x,\tau') \varphi(x) \psi(\tau') \d x \leq \int_{\tau}^{\tau'} q(t) \psi(t) \d t.
		\end{aligned}
	\]
\end{definition}

\begin{definition}
	If $\rho$ is an admissible weak solution belonging to 
	$\L{\infty}(\mathopen]0,T \mathclose[, \BV(\R, \R))$, then we will say that it is 
	$\BV$-regular. 
\end{definition}

As we pointed out before, this notion of solution is well suited for passage to the limit of a.e. convergent sequences of exact or approximate solutions. However, 
it is not so well-adapted to prove uniqueness. An equivalent notion of solution, based on explicit treatment of traces of $\rho$ at the constraint, was introduced 
by the authors of \cite{AKR2011}. This notion of solution leads to the following stability estimate.

\begin{theorem}
	\label{AWS_uniqueness_appendix}
	Fix $s^1, s^2 \in \BV([0,T], [0,\Sigma])$, $\rho_o^1,\rho_o^2 \in \L{1} \cap \BV (\R, [0,R])$ and $q^1,q^2 \in \BV([0,T], \R^+)$. Denote by $\rho^1$ 
	a $\BV$-regular solution to \eqref{AS_appendix} with data $\rho_o^1, q^1, s^1$ and $\rho^2$ an admissible weak solution to \eqref{AS_appendix} with data 
	$\rho_o^2, q^2, s^2$. Suppose that the flux functions $(t,\rho) \mapsto F(s^1(t),\rho),F(s^2(t),\rho)$ satisfy \eqref{bell_shaped_appendix}. Then for all 
    $t \in [0, T]$, we have:
	\begin{equation}
		\label{stab_appendix}
		\begin{split}
			\|\rho^1(t) - \rho^2(t)\|_{\L{1}(\R)} 
            & \leq \|\rho_o^1 - \rho_o^2\|_{\L{1}(\R)} + 2 \int_{0}^{t} |q^1(\tau) - q^2(\tau)| \d{\tau} 
            + 2 \int_{0}^{t} \| F(s^1(\tau),\cdot) - F(s^2(\tau),\cdot)\|_{\L{\infty}} \d{\tau} \\
			& + \int_{0}^{t} \left|\left| \p_{\rho} F(s^1(\tau),\cdot) - \p_{\rho}F(s^2(\tau),\cdot) \right| \right|_{\L{\infty}} \TV(\rho^1(\tau)) \d{\tau}.
		\end{split}
	\end{equation}

	In particular, Problem \eqref{AS_appendix} admits at most one $\BV$-regular solution.
\end{theorem}

\begin{proof}
	Since our interest to details lies rather on the numerical approximation point of view, we do not fully prove this statement, but we give the essential steps 
	leading to this stability result.

	$\bullet$ \textit{Definition of solution.} First, the authors of \cite{AKR2011} introduce a subset of $\R^2$ called \textit{germ}, which can be seen as the 
	set of all the possible traces of a solution to \eqref{AS_appendix}. Then, they say that $\rho$ is a solution to \eqref{AS_appendix} if it satisfies entropy 
	inequalities away from the interface -- \textit{i.e.} with $\varphi \in \Cc{\infty}(\R^* \times \R^+, \R^+)$ in the entropy inequalities -- and if the couple 
	constituted of left-side and the right-side traces of $\rho$ belongs to this so-called germ.
	
	$\bullet$ \textit{Equivalence of the two definitions.} The next step is to prove that this latter definition of solution is equivalent to Definition 
	\ref{AWS_appendix}. This part is done using good choices of test functions, see \cite[Theorem 3.18]{AKR2011} or \cite[Proposition 2.5, Theorem 2.9]{AGS2010}.

	$\bullet$ \textit{First stability estimate.} One first shows that if $s^1 = s^2$, then for all $t \in [0, T]$, one has 
	\begin{equation}
		\label{stab1_appendix}
			\|\rho^1(t) - \rho^2(t)\|_{\L{1}(\R)} \leq \|\rho_o^1 - \rho_o^2\|_{\L{1}(\R)} + 2 \int_{0}^{t} |q^1(\tau) - q^2(\tau)| \d \tau.
	\end{equation}
	
	The proof starts with the classical doubling of variables method of Kru{\v{z}}kov \cite[Theorem 1]{Kruzhkov1970} and then uses the so-called $\L{1}$-\textit{dissipativite} property, see \cite[Definition 3.1]{AKR2011} and \cite[Lemma 2.7]{AGS2010}. 

	$\bullet$ \textit{Proof of estimate} \eqref{stab_appendix}. The proof is based upon estimate \eqref{stab1_appendix} and elements borrowed from 
	\cite{BP1998, CMR2009}. Most details can be found in the proof of \cite[Theorem 2.1]{DMG2017}.
\end{proof}

\begin{remark}
	\label{on_traces}
    Though the definition of solutions with the germ explicitly involves the traces of $\rho$, we did not discuss the existence of such traces. A first way 
    to ensure such existence is to deal with $\BV$-regular solutions. That way, traces do exist and are to be understood in the sense of $\BV$ functions. 
    Outside the $\BV$ framework, existence of strong traces for solutions to \eqref{AS_appendix} is ensured provided an assumption on the fundamental diagram like 
    \eqref{non_deg}, see \cite{AM2013, NPS2018}. Finally, if one does not want to impose such a condition on the flux, (which is our case in this appendix), one 
    can follow what the authors of \cite{AKR2011} proposed (in Section 2) and consider the "singular mapping traces."
\end{remark}

\subsection{Existence of BV-regular solutions}

We now turn to the proof of the existence of $\BV$-regular solutions by the means of a finite volume scheme.

Fix $\rho_o \in \L{1}(\R, [0,R])$. For a fixed spatial mesh size $\Delta x$ and time mesh size $\Delta t$, let $x_{j} = j\Delta x $, $t^{n} = n \Delta t$. 
Define the grid cells $\bK_{j+1/2} = \mathopen]x_{j},x_{j+1}\mathclose[$ and $N \in \N^{*}$ such that $T \in [t^N, t^{N+1}\mathclose[$. We write
\[
	\R \times [0,T] \subset \bigcup_{n=0}^{N} \bigcup_{j \in \Z} \cP_{j+1/2}^{n}, \quad \cP_{j+1/2}^{n} = \bK_{j+1/2} \times [t^{n},t^{n+1}\mathclose[.
\]

Discretize the data with their mean values on each cell to obtain the sequences 
$(\rho_{j+1/2}^{0})_{j}$, $(s^{n})_{n}$ and $(q^{n})_{n}$. Following \cite{AGS2010}, 
the marching formula of the scheme is the following: for all $n \in [\![0; N-1]\!]$ 
and $j \in \Z$:
\begin{equation}
	\label{MF_appendix}
	\rho_{j+ 1/2}^{n+1} = \rho_{j+1/2}^{n} 
	- \frac{\Delta t}{\Delta x} 
	\left( \bF_{j+1}^{n}(\rho_{j+1/2}^{n},\rho_{j+3/2}^{n}) - \bF_{j}^{n}(\rho_{j-1/2}^{n},\rho_{j+1/2}^{n}) \right),
\end{equation}

where
\begin{equation}
	\label{num_flux_appendix}
	\bF_{j}^{n}(a,b) = 
	\left\{ 
		\begin{array}{cl} 
			\bF^{n}(a,b) & \text{if} \ j \neq 0 \\ 
			\min\left\{ \bF^{n}(a,b),q^{n}) \right\} & \text{if} \ j = 0, 
		\end{array} 
	\right.
\end{equation}

$\bF^{n}$ being a monotone consistent and Lipschitz numerical flux associated to 
$\rho \mapsto F(s^{n}, \rho)$. We then define
\[
	\rho_{\Delta}(x,t) = \rho_{j+1/2}^{n} \ \text{if} \ (x,t) \in \cP_{j+1/2}^{n} \quad \text{and} \quad  
	s_{\Delta}(t), q_{\Delta}(t) = s^{n}, q^n \ \text{if} \ t \in [t^{n},t^{n+1}\mathclose[.
\]

Let $\Delta = (\Delta x, \Delta t)$. For the convergence analysis, we will assume that $\Delta \to 0$, with $\lambda=\Delta t / \Delta x$, verifying the CFL 
condition
\begin{equation}
	\label{CFL_appendix}
	\lambda \underbrace{\sup_{s \in [0,\Sigma]} 
	\left( \left\| \frac{\p \bF^s}{\p a} \right\|_{\L{\infty}} + \left\| \frac{\p \bF^s}{\p b} \right\|_{\L{\infty}} \right)}_{:=L} \leq 1,
\end{equation}

where $\bF^s = \bF^s(a, b)$ is the numerical flux -- associated to $F(s,\cdot)$ -- we use in the scheme \eqref{MF_appendix}. From now, the analysis of the scheme
follows the same path as in Section \ref{Section3}. In that order, we prove that the scheme \eqref{MF_appendix}-\eqref{num_flux_appendix} is 
$\L{\infty}$ stable, satisfies discrete entropy inequalities similar to \eqref{DEI} and approximate entropy/constraint inequalities similar to 
\eqref{AEI}-\eqref{ACI}. Only the compactness for $(\rho_{\Delta})_{\Delta}$ is left to obtain since the $\Lloc{1}$ compactness for the sequences 
$(s_{\Delta})_{\Delta}$ and $(q_{\Delta})_{\Delta}$ is clear. One way to do so is to derive uniform $\BV$ bounds.

\begin{lemma}
	\label{compactness_density_appendix}
	We suppose that $\rho_o \in \L{1} \cap \BV (\R, [0,R])$ and that $q$ verifies the 
	assumption
	\begin{equation}
		\label{level_constraint_appendix}
		\exists \eps > 0, \ \forall t \in [0,T], \ \forall s \in [0,\Sigma], \ q(t) \leq \max_{\rho \in [0,R]} F(s,\rho) - \eps := q_{\eps}(s).
	\end{equation}

	Then there exists a constant $C_\eps$ depending on $\| \p_s F \|_{\L{\infty}}$ 
	such that for all $n \in [\![0; N-1]\!]$,
	\begin{equation}
		\label{BV_bound1_appendix}
		\TV(\rho_{\Delta}(t^{n+1})) 
		\leq \TV(\rho_o) + 4R + C_\eps \left( \sum_{k=0}^{n} |q^{k+1} - q^{k}| + \sum_{k=0}^{n} |s^{k+1} - s^{k}| \right),
	\end{equation}

	where $\rho_{\Delta} = \left( \rho^{n}_{j+1/2} \right)_{n,j}$ is the finite volume approximation constructed with the scheme 
	\eqref{MF_appendix}-\eqref{num_flux_appendix}, using the Godunov numerical flux when $j=0$ in \eqref{num_flux_appendix}.
\end{lemma}

\begin{proof}
	Fix $n \in [\![0; N-1]\!]$. With this set up we can follow the proofs of \cite[Section 2]{CancesSeguin2012} to obtain the following estimate:
	\[
        \sum_{j \in \Z} |\rho_{j+1/2}^{n+1} - \rho_{j-1/2}^{n+1}| 
        \leq \TV(\rho_o) + 4R + 2 \sum_{k=0}^{n} \left| \left( \widehat{\rho}_{s^{k+1}}(q^{k+1}) - \widehat{\rho}_{s^{k}}(q^{k}) \right) 
        - \left(\widecheck{\rho}_{s^{k+1}}(q^{k+1}) - \widecheck{\rho}_{s^k}(q^{k}) \right) \right|,
	\]

	where for all $k \in [\![0; n]\!]$, the couple $\left( \widehat{\rho}_{s^k}(q^{k}), \widecheck{\rho}_{s^k}(q^{k}) \right) \in [0,R]^{2}$ is uniquely 
	defined by the conditions
	\[
		F(s^{k},\widehat{\rho}_{s^k}(q^{k})) = F(s^{k},\widecheck{\rho}_{s^k}(q^{k})) = q^{k} \quad \text{and} \quad 
		\widehat{\rho}_{s^{k}}(q^{k}) > \widecheck{\rho}_{s^k}(q^{k}).
	\]

	Denote by $\Omega(\eps)$ the open subset
	\[
		\Omega(\eps) = \bigcup_{s \in [0,\Sigma]} \Omega_{s}(\eps)
	\]

	where for all $s \in [0,\Sigma]$, $\Omega_{s}(\eps) := \mathopen] \widecheck{\rho}_{s}(q_\eps (s)), \widehat{\rho}_{s}(q_\eps (s)) \mathclose[$. By Assumption 
	\eqref{level_constraint_appendix}, the continuous function $(s,\rho) \mapsto |\p_{\rho}F(s,\rho)|$ is positive on the compact subset 
	$[0,\Sigma] \times [0,R] \backslash \Omega(\eps)$. Hence, it attains its minimal value $\text{C}_{0} > 0$. Consequently, for all $s \in [0,\Sigma]$, if one 
	denotes by $I_{s} : [0,\widecheck{\rho}_{s}(q_\eps (s))] \to [0,q_{\eps}(s)]$ the increasing part of $F(s,\cdot)$, this function carries out a 
	$\Ck{1}$-diffeomorphism. Moreover,
	\[
		\forall q \in [0,q_{\eps}(s)], \ \left|(I_{s}^{-1})^{'}(q) \right| \leq \frac{1}{\text{C}_{0}}.
	\]

	Then, for all $k \in [\![0; n]\!]$,
	\[
		\begin{split}
			\left| \widecheck{\rho}_{s^{k+1}}(q^{k+1}) - \widecheck{\rho}_{s^{k}}(q^{k}) \right| 
			& = \left| (I_{s^{k+1}}^{-1})(q^{k+1}) - \widecheck{\rho}_{s^{k}}(q^{k})  \right|  \\
			& \leq \frac{1}{\text{C}_{0}} |q^{k+1} - q^{k}| + \left| (I_{s^{k+1}}^{-1})(q^{k}) - \widecheck{\rho}_{s^{k}}(q^{k}) \right| \\
			& = \frac{1}{\text{C}_{0}} |q^{k+1} - q^{k}| 
			+ \left| (I_{s^{k+1}}^{-1})(q^{k}) - (I_{s^{k+1}}^{-1}) \ \circ \ I_{s^{k+1}} \left( \widecheck{\rho}_{s^{k}}(q^{k})  \right) \right| \\
			& \leq \frac{1}{\text{C}_{0}} \left( |q^{k+1} - q^{k}| + \left| q^{k} - I_{s^{k+1}} \left( \widecheck{\rho}_{s^{k}}(q^{k}) \right) \right| \right) \\
			& = \frac{1}{\text{C}_{0}} \left( |q^{k+1} - q^{k}| 
			+ \left| F \left(s^{k}, \widecheck{\rho}_{s^{k}}(q^{k}) \right) - F\left(s^{k+1}, \widecheck{\rho}_{s^{k}}(q^{k}) \right) \right| \right)\\
			& \leq \frac{1}{\text{C}_{0}} \left( |q^{k+1} - q^{k}| + \| \p_s F \|_{\L{\infty}} |s^{k+1}-s^{k}| \right) \\
			& \leq \frac{1 + \| \p_s F \|_{\L{\infty}}}{\text{C}_{0}} \left( |q^{k+1} - q^{k}| + |s^{k+1}-s^{k}| \right).
		\end{split}
	\]

	Using the same techniques, one can show that the same inequality holds when considering 
	$\left| \widehat{\rho}_{s^{k+1}}(q^{k+1}) - \widehat{\rho}_{s^{k}}(q^{k}) \right|$. Therefore, inequality \eqref{BV_bound1_appendix} follows with
	\[
		C_\eps = 4 \times \left( \frac{1 + \| \p_s F \|_{\L{\infty}}}{\text{C}_{0}} \right).
	\]
\end{proof}

\begin{remark}
	Recall we suppose that $F : [0,\Sigma] \times [0,R]$ is continuously differentiable, but if we look in the details of the proof above, we actually need 
	$F = F(s,\rho)$ to be continuously differentiable with respect to $s$ and
	\[
		\forall s \in [0,\Sigma], \ F(s,\cdot) \in \Ck{1}([0,R] \backslash \{ \overline{\rho}_s\}, \R), \quad 
		\overline{\rho}_s = \underset{\rho \in [0,R]}{\text{argmax}} \ F(s,\rho).
	\]
\end{remark}

\begin{corollary}
	Fix $\rho_o \in \L{1} \cap \BV(\R, [0,R])$, $s \in \BV([0,T], [0,\Sigma])$ and $q \in \BV([0,T],\R^+)$. Suppose that $q$ verifies Assumption 
	\eqref{level_constraint_appendix}. Let $\rho_{\Delta} = \left( \rho^{n}_{j+1/2} \right)_{n,j}$ be the finite volume approximate solution constructed 
	with the scheme \eqref{MF_appendix}-\eqref{num_flux_appendix}, using the Godunov numerical flux when $j=0$ in \eqref{num_flux_appendix}, and any other monotone 
	consistent and Lipschitz numerical flux when $j \neq 0$. Then there exists $\rho \in \Czero([0,T], \Lloc{1}(\R, \R))$ such that
	\[
		\forall t \in [0,T], \ \rho_{\Delta}(t) \limit{\Delta}{0} \rho(t) \ \emph{in} \ \Lloc{1}(\R, \R).
	\]
\end{corollary}

\begin{proof}
	Since $s$ and $q$ have bounded variation, inequality \eqref{BV_bound1_appendix} leads to a uniform in time $\BV$ bound for the sequence 
	$\left( \rho_{\Delta} \right)_\Delta$. Then the result from \cite[Appendix]{DE2016} establish the compactness statement.
\end{proof}

\begin{theorem}
	\label{convergence_appendix}
	Fix $\rho_o \in \L{1} \cap \BV (\R, [0,R])$, $s \in \BV([0,T], [0,\Sigma])$, $F \in \Ck{1}([0,\Sigma] \times [0,R], \R)$ verifying 
	\eqref{bell_shaped_appendix} and $q \in \BV([0,T], \R^{+})$. Suppose that in \eqref{num_flux_appendix}, we use the Godunov flux when $j=0$ and any 
	other monotone consistent and Lipschitz numerical flux when $j \neq 0$. Finally, suppose that $q$ satisfies \eqref{level_constraint_appendix}. Then under 
	the CFL condition \eqref{CFL_appendix}, the scheme \eqref{MF_appendix}-\eqref{num_flux_appendix} converges to an admissible weak solution $\rho$,
	to \eqref{AS_appendix}, which is also $\BV$-regular. More precisely, there exists a 
	constant $C_\eps$ depending on $\| \p_s F \|_{\L{\infty}}$ such that
	\begin{equation}
		\forall t \in [0,T], \ \TV(\rho(t)) \leq \TV(\rho_o) + 4R + C_\eps \left( \TV(q) + \TV(s) \right).
	\label{BV_bound_appendix}
	\end{equation}
\end{theorem}

\begin{proof} 
	From the scheme \eqref{MF_appendix}, one can derive approximate entropy/constraint inequalities analogous to \eqref{AEI}-\eqref{ACI} of Section \ref{Section3}. 
	Let $\rho$ be the limit to the finite volume scheme, the compactness of $\left( \rho_{\Delta} \right)_\Delta$ coming from the last corollary. We already know 
	that $\rho \in \Czero([0,T], \Lloc{1}(\R, \R))$. By passing to the limit in the approximate entropy/constraint inequalities verified by 
	$\left( \rho_{\Delta} \right)_\Delta$ we get that $\rho$ satisfies the entropy/constraint inequalities of Definition \ref{AWS_appendix}. This shows that 
	$\rho$ is an admissible weak solution to Problem \eqref{AS_appendix}. Finally, from \eqref{BV_bound1_appendix}, the lower semi-continuity of the $\BV$ 
	semi-norm ensures that $\rho \in \L{\infty} ([0,T], \BV(\R, \R))$ and verifies \eqref{BV_bound_appendix}. This concludes the proof.
\end{proof}

\begin{corollary}
	\label{BVRS_WP_appendix}
	Fix $\rho_o \in \L{1} \cap \BV (\R, [0,R])$, $s \in \BV([0,T], [0,\Sigma])$, $F \in \Ck{1}([0,\Sigma] \times [0,R], \R)$ verifying 
    \eqref{bell_shaped_appendix} and $q \in \BV([0,T], \R^{+})$. Suppose that $q$ satisfies Assumption \eqref{level_constraint_appendix}. Then Problem 
    \eqref{AS_appendix} admits a unique $\BV$-regular solution $\rho$. Moreover, $\rho$ satisfies the bound \eqref{BV_bound_appendix}.
\end{corollary}

\begin{proof}
	Uniqueness comes from Theorem \ref{stab_appendix}, the existence and the $\BV$ bound comes from Theorem \ref{convergence_appendix}.
\end{proof}

\begin{remark}
	\label{weak_strong_appendix}
	Under the hypotheses of Corollary \ref{BVRS_WP_appendix}, if we prove the existence of another admissible weak solution $\overline{\rho}$ to 
	\eqref{AS2020_splitting} (by another method, splitting for instance), then Theorem \ref{AWS_uniqueness_appendix} ensures that $\overline{\rho} = \rho$.
\end{remark}

\newpage

{\small
  \bibliography{ref}
  \bibliographystyle{abbrv}
}

\end{document}